\documentclass[oneside]{amsart}
\usepackage[utf8]{inputenc}
\usepackage{amsmath,amstext,amsopn,amssymb, amsfonts, amsthm, mathtools}
\usepackage[dvipsnames]{xcolor}
\usepackage[marginparwidth=2.5cm]{geometry}
\usepackage{float}
\usepackage{bbm}
\usepackage{pdfsync}
\usepackage{tikz}
\usepackage{marginnote}
\usepackage{enumitem}
\usepackage{subfig}
\usepackage{caption}
\usepackage{graphicx}
\usepackage{mathrsfs}
\usepackage[ruled,vlined,algosection]{algorithm2e}
\DeclareMathAlphabet{\mathpzc}{OT1}{pzc}{m}{it}

\usepackage[pdftex,bookmarks,colorlinks,breaklinks]{hyperref}
\definecolor{dullmagenta}{rgb}{0.4,0,0.4}   
\definecolor{darkblue}{rgb}{0,0,0.4}
\definecolor{darkgreen}{rgb}{0,0.4,0}

\definecolor{cbblue}{RGB}{0,114,178}
\definecolor{cborange}{RGB}{230,159,0}
\definecolor{cbgreen}{RGB}{0,158,115}
\definecolor{cbvermillion}{RGB}{213,94,0}
\definecolor{cbsky}{RGB}{86,180,233}

\hypersetup{
  linkcolor=darkblue,
  citecolor=blue,
  filecolor=dullmagenta,
  urlcolor=darkblue
}

\newtheorem{TheoremLetter}{Theorem}
{}

\newtheorem*{definition*}{Definition}

\newtheorem{theorem}{Theorem}
\newtheorem*{theorem*}{Theorem}

\newtheorem{conjecture}{Conjecture}
\newtheorem*{conjecture*}{Conjecture}

\newtheorem*{question*}{Question}

\newtheorem{lemma}[theorem]{Lemma}
\newtheorem*{lemma*}{Lemma}

\newtheorem{proposition}[theorem]{Proposition}

\newtheorem*{corollary*}{Corollary}

\theoremstyle{definition}

\newtheorem*{remark*}{Remark}

\theoremstyle{plain}

\newtheorem*{example*}{Example}

\numberwithin{equation}{section}
\numberwithin{theorem}{section}

\makeatletter
\newcommand{\customlabel}[2]{
   \protected@write \@auxout {}{
     \string \newlabel {#1}{{#2}{\thepage}{#2}{#1}{}} }
   \hypertarget{#1}{#2}
}
\makeatother

\def\XXint#1#2#3{{\setbox0=\hbox{$#1{#2#3}{\int}$}
     \vcenter{\hbox{$#2#3$}}\kern-.5\wd0}}

\newcommand{\sh}{\operatorname{sh}}

\newcommand{\ind}{\mathbbm{1}}

\newcounter{cmt}
\newcommand{\comment}[2]{%
  \stepcounter{cmt}%
  \ifodd\value{cmt}\normalmarginpar\else\reversemarginpar\fi
  \marginnote{%
    \tikz\node[draw=blue, rounded corners=2pt, inner sep=4pt,
               text width=\dimexpr\marginparwidth-12pt\relax,
               align=left, font=\footnotesize, text=blue]{#2};%
  }[\baselineskip]%
  {\color{red}#1}%
}
\newcommand{\commentnote}[1]{%
  \stepcounter{cmt}%
  \ifodd\value{cmt}\normalmarginpar\else\reversemarginpar\fi
  \marginnote{%
    \tikz\node[draw=blue, rounded corners=2pt, inner sep=4pt,
               text width=\dimexpr\marginparwidth-12pt\relax,
               align=left, font=\footnotesize, text=red]{#1};%
  }%
}

\makeindex
\begin{document}

\title{An antichain approach to a conjecture of Zygmund}

\address{Universidad Autónoma de Madrid}
\author{Guillermo Rey}
\email{guillermo.rey@uam.es}
\thanks{Research supported in part by grants PID2022-139521NA-I00 and
  RYC2024-051323-I, both funded by MICIU/AEI/10.13039/501100011033.}
\maketitle

\begin{abstract}
  An antichain is a family of rectangles in which no member contains another.
  Given a family $\mathcal{E}$ of rectangles, define
  $h_{\mathcal{E}} = \sum_{R \in \mathcal{E}} \ind_R$.
  We show that there exist constants $c, C > 0$ such that for every sparse
  antichain $\mathcal{E}$ of dyadic rectangles in $\mathbb{R}^2$
  \begin{align*}
    \int_E \exp(c\, h_{\mathcal{E}}) \leq C |E|,
  \end{align*}
  where $E$ is the union of all the rectangles in $\mathcal{E}$.
  For general sparse families without the antichain condition, the estimate
  requires replacing $h_{\mathcal{E}}$ by $h_{\mathcal{E}}^{1/2}$,
  so antichains behave as if they lived in one dimension fewer.

  We give two applications.
  First, the dyadic Zygmund conjecture holds in dimension three: the maximal
  operator associated to dyadic rectangles with sidelengths
  $2^{m_1} \times 2^{m_2} \times 2^{\Phi(m_1, m_2)}$, where $\Phi$ is monotone
  increasing in each variable, is weak-type $L\log L$.
  This recovers a theorem of A. Córdoba.
  Second, the maximal operator of an arbitrary antichain of dyadic rectangles
  in the plane is bounded on $L^p$ with norm $\mathcal{O}(p')$, which is the
  growth of the one-parameter maximal function.
  This bound is sharp, and removing the antichain condition forces a constant
  that grows like $(p')^2$ instead.

  The proofs proceed through bounds on $k$-fold intersections: for sparse
  antichains in the plane, the $k$-wise intersection sums grow at most
  geometrically in $k$, which we prove through an $L^2$ estimate for the Gram
  matrix of the normalized indicators of the family.
  We also show that, in every dimension, the analogous exponential estimate
  for antichains is equivalent to a $k$-wise intersection bound, and implies
  the corresponding case of Zygmund's conjecture.
\end{abstract}

\tableofcontents

\section{Introduction}

Let $\mathcal{E}$ be a collection of axis-parallel rectangles in $\mathbb{R}^d$,
that is: cartesian products of $d$ one-dimensional intervals.
The associated \emph{geometric maximal function} is
\begin{align*}
  \mathcal{M}_{\mathcal{E}}f = \sup_{R \in \mathcal{E}}
    \ind_R \langle |f| \rangle_R,
\end{align*}
where $\langle f \rangle_R$ denotes the average of $f$ over $R$ with respect to
the Lebesgue measure.

When $\mathcal{E} = \mathcal{R}_d$ is the family of \emph{all} axis-parallel
rectangles, $\mathcal{M}_{\mathcal{R}_d}$ is the strong maximal function, whose
behavior near $L^1$ has been understood since Jessen, Marcinkiewicz, and Zygmund
\cite{JMZ1935}: it is weak-type $L(\log L)^{d-1}$.
We say that $\mathcal{M}_{\mathcal{E}}$ is weak-type $\psi(L)$ if
\begin{align*}
  |\{x \in \mathbb{R}^d: \, \mathcal{M}_{\mathcal{E}}f(x) > \lambda \}| \lesssim
    \int_{\mathbb{R}^d} \psi \bigg( \frac{|f(x)|}{\lambda} \bigg) \, dx,
\end{align*}
uniformly over all measurable functions $f$ and all $\lambda > 0$.
With this notation, the strong maximal function in $\mathbb{R}^d$
is weak-type $\psi_{d-1}(L)$, where $\psi_{k}(x) := x\log(e+x)^k$.

In the inequality above, $|\cdot|$ denotes the Lebesgue measure,
and we use the notation $A \lesssim B$ to mean that there exists a constant $C$,
uniform over the relevant parameters, such that $A \leq CB$.

We will only use the Lebesgue measure in this article.
Accordingly, we will often omit the variable of integration, and abbreviate
the measure of a level set $|\{x:\, g(x) > \lambda\}|$ to $|g > \lambda|$;
the inequality above then reads
$|\mathcal{M}_{\mathcal{E}}f > \lambda| \lesssim \int \psi(|f|/\lambda)$.

Zygmund conjectured that families with fewer parameters should behave better,
independent of the ambient dimension:
if the sidelengths of the rectangles in $\mathcal{E}$ are monotone increasing
functions of only $k$ parameters, then the maximal function should behave like
the $k$-dimensional strong maximal function, namely
$\mathcal{M}_{\mathcal{E}}$ should be weak-type $\psi_{k-1}(L)$.

A lot of effort has been put into making this intuition precise.
We mention \cite{
  Zygmund1967,
  Cordoba1979,
  Soria1986,
  FeffermanPipher2005,
  Stokolos2005,
  HS2022,
  HOS2023,
  ADMM2023,
  Gauvan2024}, but this is a non-exhaustive list.

A. Córdoba \cite{Cordoba1979} confirmed the conjecture for the family of
axis-parallel rectangles with sidelengths $s \times t \times \Phi(s,t)$ in
$\mathbb{R}^3$, where $\Phi : \mathbb{R}_{>0}^2 \to \mathbb{R}_{>0}$ is any
function monotone increasing in each variable.

F. Soria \cite{Soria1986} disproved the conjecture in general, constructing
monotone increasing two-parameter functions $\Phi_1$ and $\Phi_2$ such that the
family of rectangles whose sidelengths are of the form
$s \times \Phi_1(s,t) \times \Phi_2(s,t)$ behaves no better than the full
three-parameter family $\mathcal{R}_3$. See \cite{Rey2020} for further examples
in higher dimensions.

The most prominent case left open by the counterexamples, and a natural
generalization of Córdoba's result from \cite{Cordoba1979} is the situation
where, for a single $(d-1)$-parameter monotone increasing function $\Phi$, one
considers the family of axis-parallel rectangles whose sidelengths are of the
form
  $s_1 \times \dots \times s_{d-1} \times \Phi(s_1, \dots, s_{d-1}).$
In this case, Zygmund's conjecture would predict that the associated maximal
function is of weak-type $\psi_{d-2}(L)$.
In this article, we will focus on the dyadic version of this conjecture,
which we state here for reference, though we claim no authorship.
\begin{conjecture} \label{zygmund_conjecture}
  Let $\Phi : \mathbb{Z}^{d-1} \to \mathbb{Z}$ be a monotone increasing function
  in each argument, and let $\mathcal{Z}_d$ be the collection of dyadic
  rectangles in $\mathbb{R}^d$ whose sidelengths are of the form
  \begin{align*}
    2^{m_1} \times \dots \times 2^{m_{d-1}}
      \times 2^{\Phi(m_1, \dots, m_{d-1})}.
  \end{align*}
  Then $\mathcal{M}_{\mathcal{Z}_d}$ is weak-type $\psi_{d-2}(L)$.
\end{conjecture}

In this article we propose an approach to this conjecture through the study of
sparse dyadic antichains, and we carry it out completely in the first
non-trivial case of dimension three, recovering Córdoba's result.

We begin by recalling the notion of \emph{sparse} collection.
\begin{definition*}[Sparse collection]
  A collection $\mathcal{E}$ of measurable sets in $\mathbb{R}^d$ is said to be
  $\eta$-sparse if for every element $R$ in $\mathcal{E}$ there
  exists a choice of pairwise disjoint measurable sets $\{E(R)\}_{R \in
  \mathcal{E}}$ with $E(R) \subseteq R$ and $|E(R)| \geq \eta|R|$.
\end{definition*}
The \emph{shadow} of a family $\mathcal{E}$ is the set of points which
is contained in at least one of its members.
Dual to the shadow is the \emph{height function} of a family $\mathcal{E}$,
which is, at every point $x$, the number of elements in $\mathcal{E}$ that
contain $x$.
\begin{align*}
  \sh(\mathcal{E}) = \bigcup_{R \in \mathcal{E}} R
  \qquad\text{and}\qquad
  h_{\mathcal{E}} = \sum_{R \in \mathcal{E}} \ind_R.
\end{align*}
Finally, we say that $\mathcal{E}$ is an antichain if no element contains another.

We can now state our main result.
As we explain below, Theorem \ref{MT:ExpInt}
implies Conjecture \ref{zygmund_conjecture} in dimension $d=3$.
\begin{TheoremLetter} \label{MT:ExpInt}
  There exist constants $c,C > 0$ such that
  for every $\frac{1}{2}$-sparse antichain $\mathcal{E}$ of dyadic two-dimensional
  rectangles we have
  \begin{align*}
    \int_{\sh(\mathcal{E})} \exp\big(
      c h_{\mathcal{E}}
    \big) \leq C |\sh(\mathcal{E})|.
  \end{align*}
\end{TheoremLetter}
It is easy to see that one-parameter sparse collections of dyadic squares
satisfy the same inequality.
Thus, at least from the point of view of exponential-integrability,
sparse antichains of two-parameter dyadic rectangles behave as if they were
one-parameter sparse collections.

As another application, we obtain the sharp dependence on $p$ for the
$L^{p}$-operator norm of the maximal function associated to antichains of
dyadic two-dimensional rectangles.
Note that in the following theorem, we do not require that $\mathcal{E}$ be
sparse.
\begin{TheoremLetter} \label{MT:SharpLpForAntichainM}
  There exists a finite positive constant $C$ such that for
  every antichain $\mathcal{E}$ of dyadic two-dimensional rectangles,
  every $1 < p < \infty$,
  and every $f \in L^p(\mathbb{R}^2)$,
  we have
  \begin{align*}
    \|\mathcal{M}_{\mathcal{E}} f\|_{L^{p}(\mathbb{R}^2)} \leq Cp'
    \|f\|_{L^p(\mathbb{R}^2)}.
  \end{align*}
\end{TheoremLetter}
The one-parameter maximal function over all squares satisfies the same
asymptotic behavior in $p'$,
while the two-parameter strong maximal function requires a $(p')^2$.
Thus, in this setting and in parallel to the situation described above about
sparse dyadic antichains, antichains of two-dimensional dyadic rectangles behave
like one-parameter families.

A result like Theorem \ref{MT:SharpLpForAntichainM} cannot be obtained, as is
usually done, through interpolation with the weak-type endpoint.
The usual `enemies' for the weak-type boundedness of a geometric maximal
function are antichains.
One can formalize this statement, at least in the dyadic setting, as follows.
\begin{lemma*}
  Let $\mathcal{E}$ be a collection of dyadic rectangles in $\mathbb{R}^d$.
  For any integrable function $f$ and every $\lambda > 0$, we have
  \begin{align*}
    |\mathcal{M}_{\mathcal{E}} f > \lambda| =
    \sup_{\mathcal{F}} |\mathcal{M}_{\mathcal{F}} f > \lambda|,
  \end{align*}
  where the supremum is taken over all antichains $\mathcal{F} \subseteq
  \mathcal{E}$.
\end{lemma*}

Let $\mathcal{D}_d$ denote the collection of all dyadic $d$-dimensional
rectangles, that is: cartesian products of dyadic intervals.
The lemma shows that every weak-type inequality
\begin{align*}
  |\mathcal{M}_{\mathcal{F}} f > \lambda| \leq \Phi(f, \lambda),
\end{align*}
valid, uniformly over integrable $f$ and dyadic antichains $\mathcal{F}$,
is inherited by the dyadic strong maximal function
$\mathcal{M}_{\mathcal{D}_d}$.
But iteration of the easy one-dimensional maximal function bounds and testing on
indicator functions give
\begin{align*}
  \|\mathcal{M}_{\mathcal{D}_d}\|_{L^p(\mathbb{R}^d) \to L^p(\mathbb{R}^d)}
  \asymp (p')^d.
\end{align*}
In the inequality above, we have used the notation $A \asymp B$ to denote
$A \lesssim B \lesssim A$.

We now explain the connection between results like Theorem \ref{MT:ExpInt} and
Conjecture \ref{zygmund_conjecture}.
In order to do this, we first recall the geometric proof of the strong maximal
theorem of \cite{CordobaFefferman1975}.
We also describe Córdoba's proof \cite{Cordoba1979} of Conjecture
\ref{zygmund_conjecture} in dimension three, and compare it with our approach.
This will be done over the following three subsections.

\subsection{The geometric proof of the strong maximal theorem}

Covering lemmas have been a fruitful technique for proving weak-type estimates.
In \cite{Cordoba1975} Córdoba introduced a principle relating weak-type
estimates with certain Vitali-type covering lemmas.
These ideas were developed further by Córdoba and R. Fefferman, leading to the
beautiful geometric proof of the strong maximal theorem in
\cite{CordobaFefferman1975}.

Although not explicitly stated in this form, Córdoba and Fefferman use
the following result, which relates three important notions:
weak-type properties of a maximal function,
the exponential-integrability of the height function of sparse collections,
and exponential-type covering lemmas.
\begin{theorem} \label{CF:Weak-type|SparseExpInt|Covering}
  Let $\mathcal{E}$ be a countable collection of sets of finite measure.
  For every $k \geq 0$ we have
  \begin{enumerate}[label=(\alph*)]
    \item If $\mathcal{M}_{\mathcal{E}}$ is weak-type $\psi_k(L)$, then
      every $\eta$-sparse collection $\mathcal{F} \subseteq \mathcal{E}$
      satisfies
      \begin{align*}
        \int_{\sh(\mathcal{F})} \exp\Big(c h_{\mathcal{F}}^{\frac{1}{k+1}}\Big)
        \leq C |\sh(\mathcal{F})|,
      \end{align*}
      where the constants $C$ and $c$ depend on $\eta$.

    \item Suppose $k \geq 1$ and that every finite
      $\mathcal{F} \subseteq \mathcal{E}$ has a subcollection
      $\mathcal{G} \subseteq \mathcal{F}$ satisfying
      $|\sh(\mathcal{G})| \asymp |\sh(\mathcal{F})|$ as well as the
      exponential-integrability estimate
      \begin{align*}
        \int_{\sh(\mathcal{G})} \exp\Big(c h_{\mathcal{G}}^{\frac{1}{k}}\Big)
        \lesssim |\sh(\mathcal{G})|.
      \end{align*}

      Then $\mathcal{M}_{\mathcal{E}}$ is weak-type $\psi_{k}$.
  \end{enumerate}
\end{theorem}
They also provide a selection procedure which, given a finite collection
of rectangles $\mathcal{F} \subseteq \mathcal{R}_d$, produces a
subcollection $\mathcal{G} \subseteq \mathcal{F}$ of comparable measure and
satisfying the so-called $(P_2)$ condition: for every $R \in \mathcal{G}$
\begin{align} \label{p2} \tag{$P_2$}
  |R \cap \sh(\mathcal{G} \setminus \{R\})| \leq \frac{1}{2}|R|.
\end{align}
This is their selection process:
\begin{theorem}[See the proof of Theorem 2 in \cite{CordobaFefferman1975}]
  \label{CF:Selection}
  For every countable collection $\mathcal{F} \subseteq \mathcal{R}_d$ with
  $|\sh(\mathcal{F})| < \infty$ there
  exists a finite subcollection $\mathcal{G} \subseteq \mathcal{F}$ satisfying
  the \eqref{p2} condition and
  $|\sh(\mathcal{G})| \asymp |\sh(\mathcal{F})|$.
\end{theorem}

The remaining insight in \cite{CordobaFefferman1975} needed to complete the
proof is that, in the dyadic case, every \emph{slice} of a \eqref{p2} collection
is sparse.
To be precise, for each $i \in \{1, \dots, d\}$ define the projections
$\pi_i : \mathbb{R}^d \to \mathbb{R}$ and $\pi_i^* : \mathbb{R}^d \to
\mathbb{R}^{d-1}$, where $\pi_i$ is the projection that keeps only the $i$-th
coordinate, and $\pi_i^*$ is the projection that drops the $i$-th coordinate.
The following lemma is essentially contained in \cite{CordobaFefferman1975}.
\begin{lemma}
  \label{CF:slices_of_p2}
  If $\mathcal{G}$ is \eqref{p2} and consists of dyadic rectangles in
  $\mathbb{R}^d$, then for every $t \in \mathbb{R}$ and every
  $i \in \{1, \dots, d\}$
  \begin{align*}
    S_i^t(\mathcal{G}) := \{\pi_i^*(R):\, t \in \pi_i(R)\}
  \end{align*}
  is $\frac{1}{2}$-sparse, seen as a collection of $(d-1)$-dimensional dyadic
  rectangles.
  Moreover, the projection $\pi_i^*$ is a bijection between
  $\{R \in \mathcal{G}:\, t \in \pi_i(R)\}$ and $S_i^t(\mathcal{G})$.
\end{lemma}

With these ingredients, the proof of the strong maximal theorem in
\cite{CordobaFefferman1975} essentially proceeds as follows.
By the $\frac{1}{3}$-trick, we can reduce to proving that the \emph{dyadic}
strong maximal function $\mathcal{M}_{\mathcal{D}_d}$ is weak-type
$\psi_{d-1}(L)$.
Let $d \geq 2$ and assume by induction that $\mathcal{M}_{\mathcal{D}_{d-1}}$
is weak-type $\psi_{d-2}(L)$.

Fix a finite collection $\mathcal{F}$ of dyadic $d$-dimensional rectangles.
By Theorem \ref{CF:Selection}, we can extract a \eqref{p2} subcollection
$\mathcal{G} \subseteq \mathcal{F}$ with
$|\sh(\mathcal{G})| \asymp |\sh(\mathcal{F})|$.
By Fubini's theorem, Lemma \ref{CF:slices_of_p2}, and Theorem
\ref{CF:Weak-type|SparseExpInt|Covering} (a), we have
\begin{align*}
  \int_{\sh(\mathcal{G})} \exp\Big(c h_{\mathcal{G}}^{\frac{1}{d-1}}\Big)
    \lesssim |\sh(\mathcal{G})|.
\end{align*}
By Theorem \ref{CF:Weak-type|SparseExpInt|Covering} (b), this completes the
induction step.
The base case $d=1$ is the weak-type $(1,1)$ inequality of the one-dimensional
dyadic maximal function.

\subsection{Córdoba's proof of Conjecture \ref{zygmund_conjecture} in dimension
three}

In \cite{Cordoba1979}, Córdoba used these ideas to prove Conjecture
\ref{zygmund_conjecture} in dimension three.
Briefly, his proof proceeds as follows.
Start, as in the proof of the strong maximal theorem above, with a finite family
$\mathcal{F} \subseteq \mathcal{Z}_3$.
By Theorem \ref{CF:Weak-type|SparseExpInt|Covering} (b) it suffices to find
$\mathcal{G} \subseteq \mathcal{F}$ with
$|\sh(\mathcal{G})| \asymp |\sh(\mathcal{F})|$ and such that
\begin{align} \label{cordoba:exp_int}
  \int_{\sh(\mathcal{G})} \exp\big(h_{\mathcal{G}}\big)
    \lesssim |\sh(\mathcal{G})|.
\end{align}
Now, instead of applying Theorem \ref{CF:Selection}, he applies a different
selection process, one especially adapted to obtaining \eqref{cordoba:exp_int}.
In particular, he sorts the family by decreasing third side, breaking ties in
the third side by decreasing first side and then decreasing second side.
Then, in this order, he iteratively appends each rectangle $R \in \mathcal{F}$
to $\mathcal{G}$ provided that
\begin{align*}
  \frac{1}{|R|}\int_{R} \exp(h_{\mathcal{G}}) \leq \gamma,
\end{align*}
holds for some fixed constant $\gamma$.
One easily checks that the resulting family satisfies \eqref{cordoba:exp_int}.
What remains is showing that $|\sh(\mathcal{G})| \asymp |\sh(\mathcal{F})|$.
For this, suppose $R \in \mathcal{F}$ was discarded and let $\mathcal{G}_{<R}$
be all the rectangles in $\mathcal{G}$ that intersect $R$ and were processed
before it.
The key geometric fact is that every rectangle in $\mathcal{G}_{<R}$ has first
or second side at least as large as $R$'s.

If $T \in \mathcal{G}_{<R}$ had both first and second sides smaller than those
of $R$ then, by the monotonicity of $\Phi$, its third side would be at most that
of $R$; but then $T$ would have been processed after $R$, contradicting the
order.

This allows the family to be written as a union:
\begin{align*}
  \mathcal{G}_{<R} = \mathcal{H} \cup \mathcal{V},
\end{align*}
where $\mathcal{H}$ are those rectangles in $\mathcal{G}_{<R}$ with first side
at least as large as that of $R$, and similarly $\mathcal{V}$ consists of those
rectangles in $\mathcal{G}_{<R}$ with second side at least as large as that of
$R$.

Using the fact that the rectangles in $\mathcal{F}$ are dyadic, the functions
$h_{\mathcal{H}}$ and $h_{\mathcal{V}}$, restricted to $R$, depend only on $y$
and $x$ respectively, thus, if $R = I \times J \times K$
\begin{align*}
  \gamma < \frac{1}{|R|}\int_R \exp(h_{\mathcal{G}_{<R}}) \leq
  \bigg(\frac{1}{|I|} \int_I \exp(h_{\mathcal{V}})\bigg)
  \bigg( \frac{1}{|J|} \int_J\exp(h_{\mathcal{H}}) \bigg).
\end{align*}
One can then use the one-dimensional maximal function in the first and second
directions to bound the measure of the union of the discarded rectangles by
$|\sh(\mathcal{G})|$, and complete the proof.

As explained by Fefferman and Pipher in \cite{FeffermanPipher2005}, one
encounters two obstacles when trying to extend this proof to higher dimensions:
the use of higher-dimensional
maximal functions in the last step, and the fact that the splitting of
$\mathcal{G}_{<R}$ now produces functions $h_{\mathcal{B}_i}$ that, when
restricted to $R$, depend on $(d-2)$ variables. The first issue is solved in
\cite{FeffermanPipher2005}. It is unknown how to address the second issue when
$d \geq 4$.

\begin{remark*}
  Córdoba's covering lemma, like the argument sketched above, is formulated
  for dyadic rectangles, while Zygmund's original question concerns
  rectangles with arbitrary real sidelengths.
  Since monotone functions need not be doubling, rounding the sides of a
  rectangle may change the value of $\Phi$ by an arbitrarily large factor,
  so the passage between the two settings is not a routine application of
  the $\frac{1}{3}$-trick.
  In dimension three the reduction can nevertheless be carried out, using
  that the monotonicity property in Córdoba's covering lemma survives the
  rounding, see \cite{ADMM2023}.
  In this article we work exclusively in the dyadic setting.
\end{remark*}

\subsection{The approach to Zygmund's conjecture}

Our strategy towards Zygmund's conjecture replaces Córdoba's astute selection
procedure with a plain application of the \eqref{p2} selection process in
Theorem \ref{CF:Selection}, together with the following observation.
\begin{theorem} \label{slices_are_sparse_antichains}
  If $\mathcal{E}$ is a \eqref{p2} subcollection of $\mathcal{Z}_{d+1}$, then
  every slice $S_{d+1}^t(\mathcal{E})$ perpendicular to the last coordinate axis
  is a $\frac{1}{2}$-sparse antichain of dyadic $d$-dimensional rectangles.
  Moreover, for every $t \in \mathbb{R}$ and $x \in \mathbb{R}^d$,  we have
  \begin{align*}
    h_{\mathcal{E}}(x,t) = h_{S_{d+1}^t(\mathcal{E})}(x).
  \end{align*}
\end{theorem}

We will prove this theorem in Section \ref{sec:zygmund_reductions}.
Assuming, for now, that it holds, we can explain how the
exponential-integrability estimate of Theorem
\ref{MT:ExpInt} implies Conjecture
\ref{zygmund_conjecture} in dimension three.

Take a finite $\mathcal{F} \subseteq \mathcal{Z}_3$, and apply the
Córdoba-Fefferman selection process, Theorem \ref{CF:Selection}, to obtain a
\eqref{p2} subcollection $\mathcal{G} \subseteq \mathcal{F}$ satisfying
$|\sh(\mathcal{G})| \asymp |\sh(\mathcal{F})|$.
By Theorem \ref{slices_are_sparse_antichains}, every slice $S_3^t(\mathcal{G})$
is a $\frac{1}{2}$-sparse antichain of dyadic two-dimensional rectangles, thus
by Theorem \ref{MT:ExpInt}
\begin{align*}
  \int_{\sh(S_3^t(\mathcal{G}))} \exp\big( c h_{S_3^t(\mathcal{G})}\big)
    \leq C|\sh(S_3^t(\mathcal{G}))|,
\end{align*}
where the constants $c$ and $C$ are those provided by Theorem \ref{MT:ExpInt},
and are independent of $t$, $\mathcal{G}$, or $\mathcal{F}$.
By Fubini's theorem, we can integrate with respect to $t$ and obtain
\begin{align*}
  \int_{\sh(\mathcal{G})} \exp(ch_{\mathcal{G}})
  &= \int_{\mathbb{R}} \int_{\sh(S_3^t(\mathcal{G}))}
    \exp(c h_{S_3^t(\mathcal{G})}(x)) \, dx \, dt \\
  &\leq C \int_{\mathbb{R}} |\sh(S_3^t(\mathcal{G}))| \, dt
  = C|\sh(\mathcal{G})|.
\end{align*}
Applying Theorem \ref{CF:Weak-type|SparseExpInt|Covering} (b) we conclude that
$\mathcal{M}_{\mathcal{Z}_3}$ is weak-type $\psi_1(L)$.

This easy argument generalizes to arbitrary dimensions, provided we have a
suitable higher-dimensional version of Theorem
\ref{MT:ExpInt}. This observation leads to the
following conjecture, which by the argument above, would imply Conjecture
\ref{zygmund_conjecture}.
\begin{conjecture} \label{conjecture:antichain_exp_int}
  For every $d \geq 2$ there exist finite positive constants $c$ and $C$ such
  that for every $\frac{1}{2}$-sparse antichain $\mathcal{E}$ of dyadic
  $d$-dimensional rectangles, we have
  \begin{align*}
    \int_{\sh(\mathcal{E})} \exp\Big(c h_{\mathcal{E}}^{\frac{1}{d-1}}\Big)
    \leq C|\sh(\mathcal{E})|.
  \end{align*}
\end{conjecture}

The proof of Theorem \ref{MT:ExpInt} rests on the following geometric result.
\begin{theorem}
  \label{MT:kwise}
  There exists a positive finite constant $C$ such that
  for every $\frac{1}{2}$-sparse antichain $\mathcal{E}$ of two-dimensional
  dyadic rectangles and for every $k \geq 1$
  \begin{align*}
    \sum_{\substack{\mathcal{F} \subseteq \mathcal{E} \\
      \#(\mathcal{F}) = k
    }} \Big|
      \bigcap_{R \in \mathcal{F}} R
    \Big|
    \leq C^k |\sh(\mathcal{E})|.
  \end{align*}
\end{theorem}
In Section \ref{sec:zygmund_reductions} we show that Conjecture
\ref{conjecture:antichain_exp_int} is equivalent to the following more geometric
one.
\begin{conjecture} \label{conjecture:kwise}
  For every dimension $d \geq 2$ there exists a finite positive constant $C$
  such that
  for every $\frac{1}{2}$-sparse antichain $\mathcal{E}$ of $d$-dimensional
  dyadic rectangles and every $k \geq 1$
  \begin{align*}
    \sum_{\substack{\mathcal{F} \subseteq \mathcal{E} \\
      \#(\mathcal{F}) = k
    }} \Big|
      \bigcap_{R \in \mathcal{F}} R
    \Big|
    \leq C^k (k!)^{d-2} |\sh(\mathcal{E})|.
  \end{align*}
\end{conjecture}

In particular, in Section \ref{sec:zygmund_reductions}, we show.
\begin{theorem} \label{conjectures_equivalent}
  For every fixed dimension $d$, conjectures \ref{conjecture:antichain_exp_int}
  and \ref{conjecture:kwise} are equivalent, and they both imply Conjecture
  \ref{zygmund_conjecture} in dimension $d+1$.
\end{theorem}

\subsection{Structure of the paper}

Section \ref{sec:zygmund_reductions} contains the results connecting height
functions of sparse antichains to Zygmund's conjecture.
We first prove Theorem \ref{equiv:theorem}, a quantitative equivalence between
exponential-integrability estimates and $k$-wise intersection bounds, which
contains the equivalence of Conjectures \ref{conjecture:antichain_exp_int} and
\ref{conjecture:kwise} as the case $n = d-1$.
We then reprove the Córdoba-Fefferman selection theorem, Theorem
\ref{zygmund:Selection} with a variational argument.
For completeness, we also prove the slicing lemma, Lemma \ref{CF:slices_of_p2},
and deduce Theorem \ref{slices_are_sparse_antichains}.
Combining these results, we show that Conjecture
\ref{conjecture:antichain_exp_int} in dimension $d$ implies Conjecture
\ref{zygmund_conjecture} in dimension $d+1$, completing the proof of Theorem
\ref{conjectures_equivalent}.

Section \ref{sec:heights} contains the proof of Theorem \ref{MT:kwise}, which is
a particular case of the more general Theorem \ref{kwise:precise}.
We first show that the rectangles of a two-dimensional dyadic antichain
passing through a common point are totally ordered by eccentricity, and that
the measure of their intersection factors through a telescoping product of
$L^2$-normalized correlations $P(R,S)$ (Lemma \ref{chain_lemma}).
This factorization dominates the $k$-wise sums by the quadratic form
$\langle P^{k-1}v, v \rangle$, where $P$ is the Gram matrix of the
$L^2$-normalized indicators of the collection
(Lemma \ref{bound_by_bilinear}), and the sparseness of the family
bounds the operator norm of $P$ (Proposition \ref{gram}).
Theorem \ref{MT:ExpInt} then follows from Theorem \ref{kwise:precise} through
the equivalence of Theorem \ref{equiv:theorem}.
In Subsection \ref{sec:heights:sharp} we collect two examples: the family of
dyadic rectangles anchored at a common corner, which shows that the antichain
condition cannot be dropped, and the simplex construction, Theorem
\ref{sharpness:simplex}, which shows that the exponents in Theorem
\ref{MT:ExpInt} and Conjecture \ref{conjecture:antichain_exp_int} cannot be
improved.

Section \ref{sec:strong_type} is devoted to the proof of Theorem
\ref{MT:SharpLpForAntichainM}; it begins with an example showing that the
growth $p'$ of the operator norm is sharp.
To remove the sparseness assumption, we use a sparse domination theorem from
\cite{CLR2026} together with the boundedness of the maximal median operator.
This reduces the problem to the linearized operator
$\mathcal{A}_{\mathcal{E}}$ associated to a sparse antichain.
We then prove an $L^p$ restricted strong-type result for
$\mathcal{A}_{\mathcal{E}}$
using Theorem \ref{MT:ExpInt}.
Finally, the strong operator norm then follows by an interpolation theorem of
Stein-Weiss type.
The classical arguments do not track the dependence of the constants on the
exponent $p$.
Since, in our setting, this dependence is critical, we provide a full proof in
Appendix \ref{app:interpolation}.

\section{Sparse dyadic antichains and Zygmund's conjecture}
\label{sec:zygmund_reductions}

We begin by showing that Conjectures \ref{conjecture:antichain_exp_int} and
\ref{conjecture:kwise} are equivalent.
In fact, we can prove the following theorem, which is more precise and valid for
general exponents.
In particular, the equivalence part of Theorem \ref{conjectures_equivalent} is
the following theorem with $n = d-1$.
\begin{theorem}
  \label{equiv:theorem}
  Let $\mathcal{E}$ be a finite collection of sets with finite measure, and let
  $n \geq 1$ be an integer.

  Then the following two statements are equivalent, with the constants in each
  depending only on $n$ and on the constants in the other.
  \begin{enumerate}[label=(\alph*), ref=(\alph*)]
    \item \label{equiv:1}
      There exist finite positive constants $c$ and $C$  such that
      \begin{align*}
        \int_{\sh(\mathcal{E})} \exp(c h_{\mathcal{E}}^{\frac{1}{n}}) \leq
        C |\sh(\mathcal{E})|
      \end{align*}

    \item \label{equiv:2}
      There exists a finite positive constant $B$ such that for all $k \geq 1$
      \begin{align*}
        \sum_{\substack{\mathcal{F} \subseteq \mathcal{E} \\
        \#(\mathcal{F}) = k
        }} \Big|
        \bigcap_{R \in \mathcal{F}} R
        \Big|
        \leq B^k (k!)^{n-1} |\sh(\mathcal{E})|.
      \end{align*}
  \end{enumerate}
\end{theorem}

We will use the following convention, here and in the rest of the article:
constants introduced within a proof are \emph{local} to it, and the same letter
may denote a different value in another proof.
When a proof uses a constant provided by another result, we say so explicitly.
\begin{proof}
  Both directions of the equivalence rely on the following elementary identity:
  \begin{align} \label{equiv:kwise_as_binomial}
    \sum_{\substack{\mathcal{F} \subseteq \mathcal{E} \\
      \#(\mathcal{F}) = k
    }} \Big|
      \bigcap_{R \in \mathcal{F}} R
    \Big|
    = \int_{\sh(\mathcal{E})} \binom{h_{\mathcal{E}}}{k}.
  \end{align}

  We begin by proving that \ref{equiv:1} $\implies$ \ref{equiv:2} for all
  $k \geq 1$.
  Expanding the exponential as an infinite series, for every integer $k \geq 1$
  and every real number $t \geq 0$ we have
  \begin{align*}
    e^t \geq \frac{t^{nk}}{(nk)!}.
  \end{align*}
  In particular, setting $t = c h_{\mathcal{E}}^{\frac{1}{n}}$,
  integrating, and using \ref{equiv:1},  we obtain
  \begin{align*}
    \int_{\sh(\mathcal{E})} h_{\mathcal{E}}^k
    \leq \frac{(nk)!}{c^{nk}} \int_{\sh(\mathcal{E})} \exp\big(
      c h_{\mathcal{E}}^{\frac{1}{n}} \big)
    \leq \frac{(nk)!}{c^{nk}} C |\sh(\mathcal{E})|.
  \end{align*}

  By the multinomial theorem,
  \begin{align} \label{equiv:multinomial}
    \frac{(nk)!}{(k!)^n} = \binom{nk}{
      \underbrace{k, \dots, k}_{\text{$n$ times}}} \leq n^{nk},
  \end{align}
  Using $\binom{h_{\mathcal{E}}}{k} \leq \frac{h_{\mathcal{E}}^k}{k!}$,
  \eqref{equiv:kwise_as_binomial}, and \eqref{equiv:multinomial}, we conclude
  \begin{align*}
    \sum_{\substack{\mathcal{F} \subseteq \mathcal{E} \\
      \#(\mathcal{F}) = k
    }} \Big|
      \bigcap_{R \in \mathcal{F}} R
    \Big|
    \leq \frac{(nk)!}{k! \, c^{nk}} C |\sh(\mathcal{E})|
    \leq C \frac{n^{nk} (k!)^{n-1}}{c^{nk}} |\sh(\mathcal{E})|
    = C \bigg( \frac{n^n}{c^n} \bigg)^k (k!)^{n-1} |\sh(\mathcal{E})|.
  \end{align*}
  Thus, \ref{equiv:2} holds with constant
  \begin{align*}
    B = \frac{\max(C,1)n^n}{c^n}.
  \end{align*}

  We now prove \ref{equiv:2} $\implies$ \ref{equiv:1}.
  Let $\Omega = \{h_{\mathcal{E}} \geq \lambda\}$ and let $k$ be an integer
  satisfying $1 \leq k \leq \frac{\lambda}{2}$.
  We have
  \begin{align*}
    \binom{h_{\mathcal{E}}(x)}{k} \geq \frac{(\lambda - k)^k}{k!}
    \geq \frac{(\frac{\lambda}{2})^k}{k!} \qquad \text{for all $x \in \Omega$.}
  \end{align*}
  So, by Chebyshev's inequality and $k! \leq k^k$
  \begin{align*}
    |\Omega| = |h_{\mathcal{E}} \geq \lambda| &\leq
    \bigg(\frac{2}{\lambda}\bigg)^k k^k \int_{\Omega}
    \binom{h_{\mathcal{E}}}{k}.
  \end{align*}
  By \eqref{equiv:kwise_as_binomial}, \ref{equiv:2}, and $k! \leq k^k$ again, we
  arrive at
  \begin{align*}
    |\Omega| \leq \bigg(\frac{2k}{\lambda} \bigg)^k
    \sum_{\substack{\mathcal{F} \subseteq \mathcal{E} \\
      \#(\mathcal{F}) = k
    }} \Big| \bigcap_{R \in \mathcal{F}} R \Big|
    \leq \bigg(\frac{2 C_1 k}{\lambda} \bigg)^k (k!)^{n-1} |\sh(\mathcal{E})|
    &\leq \bigg(\frac{2 C_1 k^n}{\lambda} \bigg)^k |\sh(\mathcal{E})|
  \end{align*}
  where we have defined $C_1 = \max(1,B)$.

  Set $\theta := e^{-1}(2C_1)^{-\frac{1}{n}}$ and take
  $k = \lfloor \theta \lambda^{\frac{1}{n}} \rfloor$, then
  \begin{align*}
    \frac{2 C_1 k^n}{\lambda} \leq e^{-n}.
  \end{align*}
  Therefore, with this choice of $k$, and assuming $1 \leq k \leq
  \frac{\lambda}{2}$
  \begin{align*}
    |\Omega| &\leq e^{-nk} |\sh(\mathcal{E})| \\
    &\leq e^{n} \exp\Big(
      -n\theta \lambda^{\frac{1}{n}}\Big) |\sh(\mathcal{E})|.
  \end{align*}
  The constraint $1 \leq k \leq \frac{\lambda}{2}$, with this choice of $k$,
  holds for $\lambda \geq \lambda_0$, where $\lambda_0$ depends only on $C_1$
  and $n$.
  Therefore, for some finite positive constants $C_2, c_3$ that depend only on
  $C_1$ and $n$
  \begin{align*}
    |h_{\mathcal{E}} \geq \lambda| \leq C_2 \exp(-2c_3
    \lambda^{\frac{1}{n}})|\sh(\mathcal{E})|.
  \end{align*}

  The claim now follows from the layer-cake formula. Setting $c = c_3$
  \begin{align*}
    \int_{\sh(\mathcal{E})} \exp\Big(c h_{\mathcal{E}}^{\frac{1}{n}} \Big)
    &= |\sh(\mathcal{E})| + c \int_0^\infty
      |h_{\mathcal{E}} > \lambda^n| e^{c \lambda} \, d\lambda \\
    &\leq |\sh(\mathcal{E})| + C_2 c|\sh(\mathcal{E})|
      \int_0^\infty e^{-c\lambda} \, d\lambda \\
    &\leq C|\sh(\mathcal{E})|,
  \end{align*}
  for some constant $C$ depending only on $B$ and $n$.
\end{proof}

Having shown that Conjectures \ref{conjecture:antichain_exp_int} and
\ref{conjecture:kwise} are equivalent, we now turn to showing that they both
imply the dyadic case of Zygmund's conjecture, Conjecture
\ref{zygmund_conjecture}.

We begin by recalling some results, originally proved by Córdoba and Fefferman
in \cite{CordobaFefferman1975}.
Since we use a slightly different language, and to make the article as
self-contained as possible, we shall give complete proofs.

The first result states that any countable collection of rectangles has a
finite \eqref{p2} subcollection whose shadow has comparable measure.
This is Theorem \ref{CF:Selection} from the introduction.
The original proof can be found in \cite{CordobaFefferman1975}.
We give an alternative variational proof, which we believe could be of
independent interest.
It is inspired by the proof of the equivalence between Carleson and sparse
collections from Theorem 3.3 in \cite{Barron2019}.
\begin{theorem}
  \label{zygmund:Selection}
  For every countable collection $\mathcal{F} \subseteq \mathcal{R}_d$ with
  $|\sh(\mathcal{F})| < \infty$ there
  exists a finite subcollection $\mathcal{G} \subseteq \mathcal{F}$ satisfying
  the \eqref{p2} condition and
  $|\sh(\mathcal{G})| \asymp |\sh(\mathcal{F})|$.
\end{theorem}
\begin{proof}
  Since $|\sh(\mathcal{F})| < \infty$, by monotone convergence we can find
  a finite $\mathcal{F}' \subseteq \mathcal{F}$ with
  $|\sh(\mathcal{F}')| \geq \frac{1}{2}|\sh(\mathcal{F})|$, so we may assume
  that $\mathcal{F}$ itself is finite.

  Among all subcollections $\mathcal{G} \subseteq \mathcal{F}$, choose one
  maximizing the functional
  \begin{align*}
    \Delta(\mathcal{G}) := 2|\sh(\mathcal{G})| - \sum_{R \in \mathcal{G}} |R|.
  \end{align*}
  Fix $R \in \mathcal{G}$ and let $\mathcal{G}'$ be the family with $R$
  removed: $\mathcal{G}' = \mathcal{G} \setminus \{R\}$.

  Removing $R$ decreases the measure of the shadow by exactly
  $|R \setminus \sh(\mathcal{G}')|$, that is:
  \begin{align*}
    |\sh(\mathcal{G})| - |\sh(\mathcal{G}')| = |R \setminus \sh(\mathcal{G}')|.
  \end{align*}

  Thus, by maximality
  \begin{align*}
    0 \leq \Delta(\mathcal{G}) - \Delta(\mathcal{G}')
    = 2|R \setminus \sh(\mathcal{G}')| - |R|,
  \end{align*}
  which is precisely the \eqref{p2} condition.

  Next, let $S \in \mathcal{F} \setminus \mathcal{G}$.
  Adding $S$ to $\mathcal{G}$ increases the measure of the shadow by exactly
  $|S \setminus \sh(\mathcal{G})|$, so again by maximality,
  \begin{align*}
    0 \geq \Delta(\mathcal{G} \cup \{S\}) - \Delta(\mathcal{G})
    = 2|S \setminus \sh(\mathcal{G})| - |S|.
  \end{align*}
  In other words:
  \begin{align*}
    \frac{|S \cap \sh(\mathcal{G})|}{|S|} \geq \frac{1}{2}.
  \end{align*}
  Therefore, every $S \in \mathcal{F}$, selected or not, is contained in
  $\{\mathcal{M}_{\mathcal{R}_d} (\ind_{\sh(\mathcal{G})}) \geq \frac{1}{2}\}$,
  and the $L^2$-boundedness of the strong maximal function yields
  \begin{align*}
    |\sh(\mathcal{F})| \leq
    \big|\big\{\mathcal{M}_{\mathcal{R}_d} \ind_{\sh(\mathcal{G})} \geq
      \tfrac{1}{2}\big\}\big|
    \lesssim |\sh(\mathcal{G})|.
  \end{align*}
\end{proof}

The next lemma shows that every slice of a dyadic \eqref{p2} collection is
sparse.
This is Lemma \ref{CF:slices_of_p2} from the introduction.
The proof we give is essentially the one in \cite{CordobaFefferman1975}.
\begin{lemma}
  \label{zygmund:slices_of_p2}
  If $\mathcal{E}$ is \eqref{p2} and consists of dyadic rectangles in
  $\mathbb{R}^d$, then for every $t \in \mathbb{R}$ and every
  $i \in \{1, \dots, d\}$,
  the slice of $\mathcal{E}$, perpendicular to the $i$-th coordinate axis and at
  height $t$
  \begin{align*}
    S_i^t(\mathcal{E}) := \{\pi_i^*(R):\, t \in \pi_i(R)\},
  \end{align*}
  is $\frac{1}{2}$-sparse when seen as a collection of $(d-1)$-dimensional
  dyadic rectangles.

  Moreover, the projection $\pi_i^*$ is a bijection between
  $\{R \in \mathcal{E}:\, t \in \pi_i(R)\}$ and $S_i^t(\mathcal{E})$.
\end{lemma}
\begin{proof}
  Fix a \eqref{p2} family $\mathcal{E}$ consisting of $d$-dimensional dyadic
  rectangles. We will assume, without loss of generality that $i=d$.
  Set $\mathcal{E}^t = \{R \in \mathcal{E}:\, t \in \pi_d(R)\}$.

  Note that the \eqref{p2} condition implies that $\mathcal{E}$ must be an
  antichain.
  This, in turn, forces the projection
  \begin{align*}
    \pi_d^* : \mathcal{E}^t \to S_d^t(\mathcal{E})
  \end{align*}
  to be a bijection between $\mathcal{E}^t$ and $S_d^t(\mathcal{E})$.
  Indeed, the collection $S_d^t(\mathcal{E})$ is defined precisely as the image
  of $\mathcal{E}^t$ through $\pi_d^*$, so we just need to show it is injective.
  To see this, suppose by contradiction that two different rectangles
  $R, S \in \mathcal{E}^t$
  have the same projection $\pi_d^*(R) = B = \pi_d^*(S)$.
  We then have
  \begin{align*}
    R = B \times \pi_d(R) \quad\text{and}\quad S = B \times \pi_d(S).
  \end{align*}
  Since $t \in \pi_d(R) \cap \pi_d(S)$ and the intervals are dyadic, we must
  have $\pi_d(R) \subseteq \pi_d(S)$ or $\pi_d(S) \subseteq \pi_d(R)$.
  But then, since $R$ and $S$ share the same base $B$, we must have
  \begin{align*}
    R \subseteq S \quad\text{or}\quad S \subseteq R.
  \end{align*}
  In either case, we reach a contradiction since the family $\mathcal{E}$ is an
  antichain.

  We now turn to the sparsity claim.
  We have just seen that, for every $\overline{R} \in S_d^t(\mathcal{E})$, there
  must exist a unique $R \in \mathcal{E}^t$ such that
  $\overline{R} = \pi_d^*(R)$. Define
  \begin{align*}
    E_d(R) := R \setminus \sh(\{S \in \mathcal{E}^t \setminus \{R\}:\,
      \pi_d(S) \supseteq \pi_d(R) \}).
  \end{align*}
  Note that, since the intervals are all dyadic, the set $E_d(R)$ factors as
  \begin{align*}
    E_d(R) = F(R) \times \pi_d(R)
  \end{align*}
  for some $F(R) \subseteq \overline{R}$.
  The \eqref{p2} condition implies that $|E_d(R)| \geq \frac{1}{2}|R|$, thus
  \begin{align*}
    |F(R)| \geq \frac{1}{2}|\overline{R}|.
  \end{align*}

  We set, as in the definition of $\frac{1}{2}$-sparse,
  $E(\overline{R}) := F(R)$.
  The only thing left to show is that the collection $\{E(\overline{R})\}$ is
  pairwise disjoint.

  To this end, let $\overline{R} \neq \overline{S}$ be any two distinct elements
  in $S_d^t(\mathcal{E})$, and let $R$ and $S$ be the unique rectangles in
  $\mathcal{E}^t$ defined by their projections $\pi_d^*(R) = \overline{R}$
  and $\pi_d^*(S) = \overline{S}$.
  Assume, without loss of generality, that $\pi_d(R) \subseteq \pi_d(S)$.
  Then, by the way we constructed the sets $E_d$, we have
  $E_d(R) \cap S = \emptyset$. Since $\pi_d(R) \subseteq \pi_d(S)$, this implies
  $F(R) \cap F(S) = \emptyset$, and this completes the proof.
\end{proof}

Theorem \ref{slices_are_sparse_antichains} is now an easy consequence of the
lemma above. We state the theorem again for convenience.
\begin{theorem}
  \label{zygmund:slices_are_sparse_antichains}
  If $\mathcal{E}$ is a \eqref{p2} subcollection of $\mathcal{Z}_{d+1}$, then
  every slice $S_{d+1}^t(\mathcal{E})$ perpendicular to the last coordinate axis
  is a $\frac{1}{2}$-sparse antichain of dyadic $d$-dimensional rectangles.
  Moreover, for every $t \in \mathbb{R}$ and $x \in \mathbb{R}^d$,  we have
  \begin{align} \label{zygmund:slices_are_sparse_antichains:eq}
    h_{\mathcal{E}}(x,t) = h_{S_{d+1}^t(\mathcal{E})}(x).
  \end{align}
\end{theorem}
\begin{proof}
  Fix a \eqref{p2} collection $\mathcal{E} \subseteq \mathcal{Z}_{d+1}$ and any
  $t \in \mathbb{R}$.

  Lemma \ref{zygmund:slices_of_p2} already implies that $S_{d+1}^t(\mathcal{E})$
  is a $\frac{1}{2}$-sparse collection.
  The bijection claim implies \eqref{zygmund:slices_are_sparse_antichains:eq}.
  Therefore, we only need to show that $S_{d+1}^t(\mathcal{E})$ is \emph{itself}
  an antichain.

  Let $R,S \in S_{d+1}^t(\mathcal{E})$, and define as in Lemma
  \ref{zygmund:slices_of_p2}, the unique rectangles
  $\widehat{R}, \widehat{S} \in \mathcal{E}$ such that
  $t \in \pi_{d+1}(\widehat{R}) \cap \pi_{d+1}(\widehat{S})$ and
  \begin{align*}
    \pi_{d+1}^*(\widehat{R}) = R \quad\text{and}\quad
    \pi_{d+1}^*(\widehat{S}) = S.
  \end{align*}

  Suppose $R \subseteq S$, then by the monotonicity of $\Phi$, we must have
  \begin{align*}
    \pi_{d+1}(\widehat{R}) \subseteq \pi_{d+1}(\widehat{S}).
  \end{align*}
  This implies that $\widehat{R} \subseteq \widehat{S}$.
  Since every \eqref{p2} collection is an antichain, we conclude
  $\widehat{R} = \widehat{S}$, which in turn implies $R = S$.
\end{proof}

We now show that Conjecture \ref{conjecture:antichain_exp_int} implies the
dyadic Zygmund conjecture \ref{zygmund_conjecture}.
\begin{theorem}
  Suppose Conjecture \ref{conjecture:antichain_exp_int} holds for a fixed
  dimension $d$. Then, Conjecture \ref{zygmund_conjecture} holds for dimension
  $d+1$.
\end{theorem}
\begin{proof}
  By standard density arguments, it suffices to show, for
  every positive function $f \in L^2(\mathbb{R}^{d+1})$ and every $\lambda > 0$
  the inequality
  \begin{align*}
    |\{x \in \mathbb{R}^{d+1}:\,
      \mathcal{M}_{\mathcal{Z}_{d+1}}f(x) > \lambda\}| \lesssim
      \int_{\mathbb{R}^{d+1}} \psi_{d-1}\bigg(\frac{f}{\lambda}\bigg).
  \end{align*}
  By homogeneity, we can further assume that $\lambda = 1$.
  Set
  $\mathcal{E} = \{R \in \mathcal{Z}_{d+1}:\, \langle f \rangle_R > 1\}$.
  Then
  \begin{align*}
  \Omega := \{x \in \mathbb{R}^{d+1}:\,
    \mathcal{M}_{\mathcal{Z}_{d+1}}f(x) > 1\} = \sh(\mathcal{E}).
  \end{align*}
  By the $L^2$-boundedness of the strong maximal function, we have
  $|\sh(\mathcal{E})| < \infty$, thus by Theorem \ref{CF:Selection}, there
  exists a finite $\mathcal{G} \subseteq \mathcal{E}$ which is \eqref{p2} and
  such that $|\sh(\mathcal{E})| \leq C_1|\sh(\mathcal{G})|$, where the constant
  $C_1$ is purely dimensional.

  For every $R \in \mathcal{G}$ we have $|R| \leq \int_R f$, thus
  \begin{align*}
    |\Omega| \leq C_1 \sum_{R \in \mathcal{G}} |R|
    \leq C_1 \int_{\mathbb{R}^{d+1}} f h_{\mathcal{G}}.
  \end{align*}
  By Young's inequality we have, for every real $a,b \geq 0$,
  $ab \leq \psi_{d-1}(a) + \exp\big(d\, b^{\frac{1}{d-1}}\big).$
  Let $c$ and $C$ be the constants from Conjecture
  \ref{conjecture:antichain_exp_int}.
  Then, for a constant $C_2$ depending only on the dimension, $c$, and $C$,
  we have
  \begin{align*}
    C_1 ab \leq C_2 \psi_{d-1}(a) +
      \frac{1}{2C} \exp\Big(cb^{\frac{1}{d-1}}\Big).
  \end{align*}
  Setting $a = f$, $b = h_{\mathcal{G}}$, and integrating
  \begin{align*}
    |\Omega| \leq C_2 \int_{\mathbb{R}^{d+1}} \psi_{d-1}(f)
    + \frac{1}{2C} \int_{\sh(\mathcal{G})} \exp\Big(
      c h_{\mathcal{G}}^{\frac{1}{d-1}}
    \Big).
  \end{align*}
  By Theorem \ref{zygmund:slices_are_sparse_antichains}, Fubini's theorem, and
  Conjecture \ref{conjecture:antichain_exp_int}, we can estimate the last term
  as follows:
  \begin{align*}
    \frac{1}{2C}\int_{\sh(\mathcal{G})} \exp\Big(
      c h_{\mathcal{G}}(x,t)^{\frac{1}{d-1}}
    \Big) \, dx \, dt
    &= \frac{1}{2C} \int_{\mathbb{R}} \int_{\sh(S_{d+1}^t(\mathcal{G}))}
      \exp\Big(
        c h_{S_{d+1}^t(\mathcal{G})}(x)^{\frac{1}{d-1}}
      \Big) \, dx \, dt \\
    &\leq \frac{1}{2} \int_{\mathbb{R}} |\sh(S_{d+1}^t(\mathcal{G}))| \, dt
    = \frac{1}{2}|\sh(\mathcal{G})| \leq \frac{1}{2}|\Omega|.
  \end{align*}
  Thus,
  \begin{align*}
    |\Omega| \leq 2C_2 \int_{\mathbb{R}^{d+1}} \psi_{d-1}(f),
  \end{align*}
  and we are done.
\end{proof}

\section{Height functions of sparse dyadic antichains} \label{sec:heights}

In this section, we will call $\mathcal{E}$ a \emph{dyadic antichain} whenever
it is an antichain of dyadic rectangles in an arbitrary dimension.
All collections in this section will, unless explicitly stated otherwise, be at
most countable.

The purpose of this section is to prove the following result,
which is a more general version of Theorem \ref{MT:kwise}.

\begin{theorem} \label{kwise:precise}
  If $\mathcal{E}$ is a two-dimensional $\eta$-sparse dyadic antichain, then
  for every $k \geq 1$
  \begin{align} \label{kwise:ineq}
    \sum_{\substack{\mathcal{F} \subseteq \mathcal{E} \\
      \#(\mathcal{F}) = k
    }} \Big|
      \bigcap_{R \in \mathcal{F}} R
    \Big|
    \leq B^{k-1} \sum_{R \in \mathcal{E}} |R|,
  \end{align}
  where $B$ is a positive finite constant depending only on $\eta$.
\end{theorem}
If $\mathcal{E}$ is $\eta$-sparse, then it satisfies the Carleson packing
condition:
\begin{align*}
  \sum_{R \in \mathcal{E}} |R| \leq \eta^{-1} |\sh(\mathcal{E})|,
\end{align*}
thus, Theorem \ref{MT:kwise} follows from Theorem \ref{kwise:precise}.
In turn, using Theorem \ref{equiv:theorem} with $n=1$, together with the
monotone convergence theorem, Theorem \ref{MT:ExpInt}
follows from Theorem \ref{kwise:precise}.

\subsection{The order structure of two-dimensional antichains}

We shall focus now on two dimensions and study how dyadic rectangles can
intersect. We begin by introducing a partial order that will become particularly
useful in two dimensions.
\begin{definition*}
  For any two dyadic rectangles $R$ and $S$, define
  \begin{align*}
    R \prec S \iff \pi_1(R) \subsetneq \pi_1(S)
    \quad\text{and}\quad
    \pi_2(R) \supsetneq \pi_2(S).
  \end{align*}
\end{definition*}
When $R$ and $S$ belong to the same antichain, $R \cap S \neq \emptyset$ implies
that $R$ and $S$ are related by $\prec$, or in other words:
they must intersect in the shape of a cross.

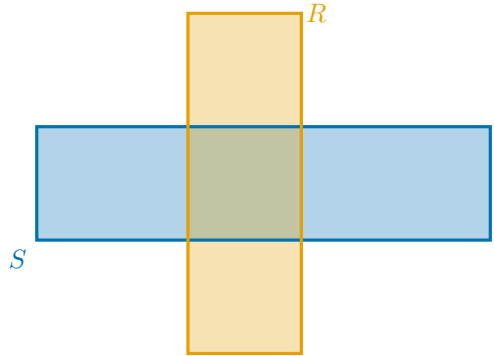
\begin{figure}[htpb]
  \centering
  \begin{tikzpicture}
    \fill[cbblue, fill opacity=0.3] (0,1.5) rectangle (6,3);
    \fill[cborange, fill opacity=0.3] (2,0) rectangle (3.5,4.5);

    \draw[cbblue, very thick] (0,1.5) rectangle (6,3);
    \draw[cborange, very thick] (2,0) rectangle (3.5,4.5);

    \node[cborange] at (3.7, 4.5) {$R$};
    \node[cbblue] at (-0.25, 1.25) {$S$};
  \end{tikzpicture}
  \caption{Two dyadic rectangles intersecting in a cross.
  In this example $R \prec S$.}
  \label{fig:cross}
\end{figure}
Assume that $R$ and $S$ are two different, intersecting, dyadic rectangles,
and assume that neither contains the other.
The sides of $R$ and $S$ are dyadic, therefore $\pi_1(R) \subseteq \pi_1(S)$ or
$\pi_1(R) \supseteq \pi_1(S)$, and similarly for $\pi_2$.
Suppose $\pi_1(R) \subseteq \pi_1(S)$. The two possibilities for $\pi_2$ are:
\begin{align*}
  \pi_2(R) \subseteq \pi_2(S) \quad\text{or}\quad \pi_2(S) \subseteq \pi_2(R).
\end{align*}
The first case violates the antichain condition, since then we would have $R
\subseteq S$.

Therefore if $R$ and $S$ intersect and
$\pi_1(R) \subseteq \pi_1(S)$, then we must also have $\pi_2(R) \supseteq
\pi_2(S)$. Moreover, if $\pi_1(R) = \pi_1(S)$, then we again have containment,
and similarly for $\pi_2$. This shows that, whenever $R$ and $S$ intersect they
must be strongly ordered by $\prec$.

Theorem \ref{kwise:precise} talks about intersections of $k$ rectangles.
Thus, we now wish to study the situation of the intersection of $k$ rectangles
taken from a dyadic antichain. Suppose we sort them by $\prec$:
\begin{align*}
  R_1 \prec R_2 \prec \dots \prec R_k.
\end{align*}
The intersection of the $k$ rectangles is simply
\begin{align*}
  R_1 \cap \dots \cap R_k = R_1 \cap R_k
  \implies |R_1 \cap \dots \cap R_k| = |\pi_1(R_1)||\pi_2(R_k)|.
\end{align*}
While true, this loses the information that all the other $(k-2)$ rectangles
could provide.
Also, it treats the first and second projection differently. It would be
convenient to have an expression for $|R_1 \cap \dots \cap R_k|$ which is as
symmetric as possible.

In order to motivate the next result, where this idea is carried out fully, we
will first focus on the case of just three intersecting rectangles.
Suppose we have $R \prec S \prec T$, then
\begin{align*}
  |R \cap S \cap T| &= |\pi_1(R)||\pi_2(T)| \\
  &= |\pi_1(R)|
     {\color{cbblue}|\pi_2(S)| |\pi_2(S)|^{-1} }
    |\pi_2(T)| \\
  &= |R \cap S| { \color{cbblue} |\pi_2(S)|^{-1}}
      {\color{cbvermillion} |\pi_1(S)|} |\pi_2(T)|
      {\color{cbvermillion} |\pi_1(S)|^{-1}} \\
  &= |R \cap S| |S \cap T| \frac{1}{|S|}.
\end{align*}
We can symmetrize this computation slightly and obtain
\begin{align*}
  |R \cap S \cap T| =
    \frac{|R \cap S|}{\sqrt{|S|}}
    \frac{|S \cap T|}{\sqrt{|S|}}
    =
    \sqrt{|R|} \frac{|R \cap S|}{\sqrt{|R||S|}}
    \frac{|S \cap T|}{\sqrt{|S||T|}} \sqrt{|T|}
\end{align*}

The following lemma, which is the only time when we need the dimension to be
two, generalizes this computation.
\begin{lemma} \label{chain_lemma}
  Let $\mathcal{E}$ be a finite dyadic antichain, and suppose $\mathcal{F}
  \subseteq \mathcal{E}$ is a subcollection of rectangles with
  non-empty common intersection.
  Then $\mathcal{F}$ is strongly ordered by $\prec$,
  and if we write the elements of $\mathcal{F}$ as
  \begin{align*}
    R_1 \prec R_2 \prec \dots \prec R_k,
  \end{align*}
  then
  \begin{align} \label{chain_lemma:telescope}
    \Bigg|\bigcap_{R \in \mathcal{F}} R \Bigg| = |R_1|^{\frac{1}{2}} \Bigg(
      \prod_{i=1}^{k-1} P(R_i, R_{i+1})
    \Bigg) |R_k|^{\frac{1}{2}},
  \end{align}
  where for any two rectangles $R$ and $S$, $P(R,S)$ is the $L^2$-normalized
  ``correlation''
  \begin{align*}
    P(R,S) := \frac{|R \cap S|}{|R|^{\frac{1}{2}} |S|^{\frac{1}{2}}}.
  \end{align*}
\end{lemma}
\begin{proof}
  We have already shown that $\mathcal{F}$ is strongly ordered.

  Note that for $\{R_i\}$ increasing with respect to $\prec$ we have
  \begin{align*}
    P(R_i, R_{i+1}) &= \frac{|R_i \cap R_{i+1}|}{\sqrt{|R_i| |R_{i+1}|}} \\
    &= \frac{|\pi_1(R_i)| |\pi_2(R_{i+1})|}{\sqrt{|R_i| |R_{i+1}|}} \\
    &= \bigg(\frac{|\pi_1(R_i)|}{|\pi_2(R_i)|}\bigg)^{\frac{1}{2}}
       \bigg(\frac{|\pi_1(R_{i+1})|}{|\pi_2(R_{i+1})|}\bigg)^{-\frac{1}{2}}.
  \end{align*}
  Therefore, the product in \eqref{chain_lemma:telescope} telescopes:
  \begin{align*}
    \prod_{i=1}^{k-1} P(R_i, R_{i+1}) =
      \bigg(\frac{|\pi_1(R_1)|}{|\pi_2(R_1)|}\bigg)^{\frac{1}{2}}
      \bigg(\frac{|\pi_1(R_k)|}{|\pi_2(R_k)|}\bigg)^{-\frac{1}{2}},
  \end{align*}
  and we have
  \begin{align*}
    |R_1|^{\frac{1}{2}} \Bigg(
      \prod_{i=1}^{k-1} P(R_i, R_{i+1})
    \Bigg) |R_k|^{\frac{1}{2}} = |\pi_1(R_1)||\pi_2(R_k)|,
  \end{align*}
  which is what we needed to prove.
\end{proof}

If we use this lemma to try to compute the left-hand side of \eqref{kwise:ineq}
we arrive at the following expression
\begin{align*}
  \sum_{\substack{
    \mathcal{F} \subseteq \mathcal{E} \\
    \#(\mathcal{F}) = k
  }} \Bigg|\bigcap_{R \in \mathcal{F}} R \Bigg| =
  \sum_{R_1 \prec \dots \prec R_k} |R_1|^{\frac{1}{2}} \Bigg(
      \prod_{i=1}^{k-1} P(R_i, R_{i+1})
    \Bigg) |R_k|^{\frac{1}{2}},
\end{align*}
where the sum is taken over all ordered $k$-element subcollections of
$\mathcal{E}$.
If, instead, we sum over all possible (ordered and possibly with repetitions)
choices of $k$ rectangles in $\mathcal{E}$, we arrive at the analytically
simpler inequality
\begin{align*}
  \sum_{\substack{
    \mathcal{F} \subseteq \mathcal{E} \\
    \#(\mathcal{F}) = k
  }} \Bigg|\bigcap_{R \in \mathcal{F}} R \Bigg| \leq
  \sum_{(R_1, \dots, R_k) \in \mathcal{E}^k}
  |R_1|^{\frac{1}{2}} \Bigg(
      \prod_{i=1}^{k-1} P(R_i, R_{i+1})
    \Bigg) |R_k|^{\frac{1}{2}}.
\end{align*}

Let us focus on the $k=2$ case first.
The right-hand side of the above inequality is simply
\begin{align*}
  \sum_{R,S \in \mathcal{E}} |R|^{\frac{1}{2}} P(R,S) |S|^{\frac{1}{2}} =:
  \circledast.
\end{align*}
This expression looks very similar to a bilinear form. Indeed, if we treat $P$
as an $\mathcal{E}\times \mathcal{E}$ matrix, then it is a real, symmetric, and
for $v \in \ell^2(\mathcal{E})$ given by $v_R = |R|^{\frac{1}{2}}$, we have
\begin{align*}
  \circledast = \langle Pv, v \rangle.
\end{align*}
Higher values of $k$ can be understood as the same expression but with higher
powers of $P$. This is the content of the next lemma.
\begin{lemma} \label{bound_by_bilinear}
  If $\mathcal{E}$ is a dyadic antichain, and we treat $P$ as the Gram matrix of
  a bilinear form like above, then
  \begin{align*}
    \sum_{\substack{
      \mathcal{F} \subseteq \mathcal{E} \\
      \#(\mathcal{F}) = k
    }} \Bigg|\bigcap_{R \in \mathcal{F}} R \Bigg| \leq
    \langle P^{k-1} v, v \rangle,
  \end{align*}
  where $v$ is the $\ell^2(\mathcal{E})$ vector given by $v_R =
  |R|^{\frac{1}{2}}$.
\end{lemma}
\begin{proof}
  Written from the point of view of the Hilbert space of $\ell^2$ sequences over
  $\mathcal{E}$, by Lemma \ref{chain_lemma} we have to show
  \begin{align*}
    \sum_{(R_1, \dots, R_k) \in \mathcal{E}^k}
    v_{R_1} \Bigg(
    \prod_{i=1}^{k-1} P(R_i, R_{i+1})
    \Bigg) v_{R_k} = \langle P^{k-1} v, v \rangle.
  \end{align*}
  This is a standard computation from linear algebra: one can show by
  induction on $k$ that for any symmetric $N \times N$ matrix $A$ we have
  \begin{align*}
    \langle A^{k-1}u, v \rangle = \sum_{i_1, \dots, i_k \in [N]}
    \Bigg( \prod_{j=1}^{k-1} A(i_j, i_{j+1}) \Bigg) u_{i_1} v_{i_k},
  \end{align*}
  where $[N] = \{1, \dots, N\}$.
  Applying this with $A = P$ and $u = v$ yields the desired identity.
\end{proof}

Note that this Lemma is almost enough to prove Theorem \ref{kwise:precise}.
Indeed, suppose we had the following result:
\begin{proposition} \label{gram}
  For every $\eta \in (0,1)$,
  every dimension $d \geq 2$,
  and every finite dyadic antichain $\mathcal{E}$ that is $\eta$-sparse,
  we have
  \begin{align*}
    \|Pu\|_{\ell^2(\mathcal{E})} \leq \frac{4^d}{\eta}
    \|u\|_{\ell^2(\mathcal{E})}
  \end{align*}
  for all $u \in \ell^2(\mathcal{E})$.
\end{proposition}
Then, Theorem \ref{kwise:precise} follows.
\begin{proof}[Proof of Theorem \ref{kwise:precise}]
  Suppose $\mathcal{E}$ is a dyadic antichain, then by Lemma
  \ref{bound_by_bilinear} and Proposition \ref{gram} we have
  \begin{align*}
    \sum_{\substack{\mathcal{F} \subseteq \mathcal{E} \\
    \#(\mathcal{F}) = k
    }} \Big|
    \bigcap_{R \in \mathcal{F}} R
    \Big| \leq \langle P^{k-1} v, v \rangle &\leq B^{k-1} \|v\|_{\ell^2}^2 \\
    &= B^{k-1} \sum_{R \in \mathcal{E}} |R|,
  \end{align*}
  where $B = 4^d / \eta$ and, like before, $v_R := |R|^{\frac{1}{2}}$.
\end{proof}

Thus, we only have to prove Proposition \ref{gram}.
\begin{proof}[Proof of Proposition \ref{gram}]
  The matrix $P$ is real and symmetric.
  In fact, if for every $T \in \mathcal{E}$ we set
  \begin{align*}
    e_T := \frac{\ind_T}{|T|^{\frac{1}{2}}},
  \end{align*}
  then
  \begin{align*}
    P(R,S) = \frac{|R \cap S|}{\sqrt{|R||S|}}
    = \int_{\mathbb{R}^d} e_R e_S = \langle e_R, e_S \rangle.
  \end{align*}
  Thus $P$ is the Gram matrix of the system $\{e_R\}_{R \in \mathcal{E}}$,
  and in particular it is positive semidefinite. For such a matrix the
  operator norm is attained by the associated quadratic form,
  \begin{align*}
    \|P\|_{\ell^2(\mathcal{E}) \to \ell^2(\mathcal{E})}
    = \sup_{\|u\|_{\ell^2(\mathcal{E})} = 1} \langle Pu, u \rangle.
  \end{align*}
  This observation allows us to, when expanding the sum, decouple the
  interaction among the rectangles:
  \begin{align*}
    \langle Pu, u \rangle = \sum_{R, S \in \mathcal{E}}
    \langle e_R, e_S \rangle u_R u_S
    = \bigg\| \sum_{R \in \mathcal{E}} u_R e_R
      \bigg\|_{L^2(\mathbb{R}^d)}^2.
  \end{align*}
  So it suffices to bound $\|\sum_{R \in \mathcal{E}} u_R e_R\|_{L^2}$,
  which is the adjoint formulation of the following sparse square function
  bound:
  if $\mathcal{E}$ is an $\eta$-sparse collection of axis-parallel dyadic
  $d$-dimensional rectangles, then
  \begin{align*}
    \sum_{R \in \mathcal{E}} \langle f, e_R \rangle^2 \leq \frac{4^d}{\eta}
    \|f\|_{L^2(\mathbb{R}^d)}^2.
  \end{align*}
  for every non-negative $f \in L^2(\mathbb{R}^2)$.
  The proof of this result is standard, but we include it for completeness.

  Using the definition of $e_R$:
  \begin{align*}
    \sum_{R \in \mathcal{E}} \langle f, e_R \rangle^2 &=
      \sum_{R \in \mathcal{E}} \frac{1}{|R|} \bigg( \int_R f \bigg)^2 \\
      &= \sum_{R \in \mathcal{E}} \bigg( \frac{1}{|R|} \int_R f \bigg)^2 |R| \\
      &\leq \eta^{-1} \sum_{R \in \mathcal{E}} \bigg( \frac{1}{|R|} \int_R f
      \bigg)^2 |E(R)|,
  \end{align*}
  where we have used the definition of $\eta$-sparse collection in the last
  line.

  Since the sets $\{E(R)\}$ are pairwise disjoint, and $E(R) \subseteq R$, we
  can bound the expression in the last line with the dyadic
  strong maximal function $\mathcal{M}_{\mathcal{D}_d}$.
  For every $p > 1$ and every $f \in L^p(\mathbb{R}^d)$ we have
  \begin{align}
    \label{doob}
    \|\mathcal{M}_{\mathcal{D}_d}f\|_{L^p(\mathbb{R}^d)} \leq (p')^d
    \|f\|_{L^p(\mathbb{R}^d)},
  \end{align}
  which can be proved iterating the corresponding one-dimensional bound.

  Using this inequality with $p=2$,
  \begin{align*}
    \sum_{R \in \mathcal{E}} \langle f, e_R \rangle^2 &\leq
    \eta^{-1} \int_{\mathbb{R}^d} (\mathcal{M}_{\mathcal{D}_d} f)^2 \\
    &\leq \frac{4^{d}}{\eta} \|f\|_{L^2}^2.
  \end{align*}

  It only remains to finish the estimate of $\langle Pu, u \rangle$.
  By duality and the Cauchy-Schwarz inequality,
  \begin{align*}
    \bigg\| \sum_{R \in \mathcal{E}} u_R e_R \bigg\|_{L^2}
    = \sup_{\|f\|_{L^2} = 1} \sum_{R \in \mathcal{E}} u_R
      \langle f, e_R \rangle
    \leq \|u\|_{\ell^2(\mathcal{E})} \sup_{\|f\|_{L^2} = 1}
      \bigg( \sum_{R \in \mathcal{E}} \langle f, e_R \rangle^2
      \bigg)^{\frac{1}{2}},
  \end{align*}
  so that $\langle Pu, u \rangle \leq \frac{4^d}{\eta}
  \|u\|_{\ell^2(\mathcal{E})}^2$.
\end{proof}

\subsection{Sharpness} \label{sec:heights:sharp}

In this subsection we consider two classes of examples that explain the way in
which Theorem \ref{MT:ExpInt} is sharp.
The examples are naturally $d$-dimensional, thus they also address the sharpness
of Conjecture \ref{conjecture:antichain_exp_int} if true.
In view of Theorem \ref{equiv:theorem}, the sharpness claims transfer to Theorem
\ref{MT:kwise} and Conjecture \ref{conjecture:kwise}.

\subsubsection{The antichain condition}
The first example shows that the antichain condition cannot be dropped while
maintaining the exponent of $h_{\mathcal{E}}$.

In particular, we prove
\begin{theorem}
  Fix $\alpha > 0$ and suppose there exist finite positive constants
  $c$ and $C$ such that
  \begin{align*}
    \int_{\sh(\mathcal{E})} \exp\big( c h_{\mathcal{E}}^\alpha \big)
    \leq C |\sh(\mathcal{E})|
  \end{align*}
  holds for all $\frac{1}{2}$-sparse, not necessarily antichain, collections of
  dyadic rectangles in $\mathbb{R}^d$.
  Then $\alpha \leq \frac{1}{d}$.
\end{theorem}
\begin{proof}
  Consider the specific example of family given by all dyadic rectangles
  that contain the origin, and are contained in $[0,1)^d$:
  \begin{align*}
    \mathcal{F} &= \big\{R:\,
    \text{$R$ is a dyadic rectangle such that
    $0 \in R$ and $R \subseteq [0,1)^d$
    }
    \big\} \\
    &= \big\{\big[0,2^{-m_1}\big) \times \dots \times \big[0,2^{-m_d}\big):\,
    m \in \mathbb{Z}_{\geq 0}^d \big\}.
  \end{align*}
  See Figure \ref{figure:chains_example}.
  The family $\mathcal{F}$ is $2^{-d}$-sparse, indeed one can take
  \begin{align*}
    E\Bigg(\prod_{i=1}^d \big[0,2^{-m_i}\big)\Bigg)
    := \prod_{i=1}^d \big[2^{-m_i-1}, 2^{-m_i}\big).
  \end{align*}
  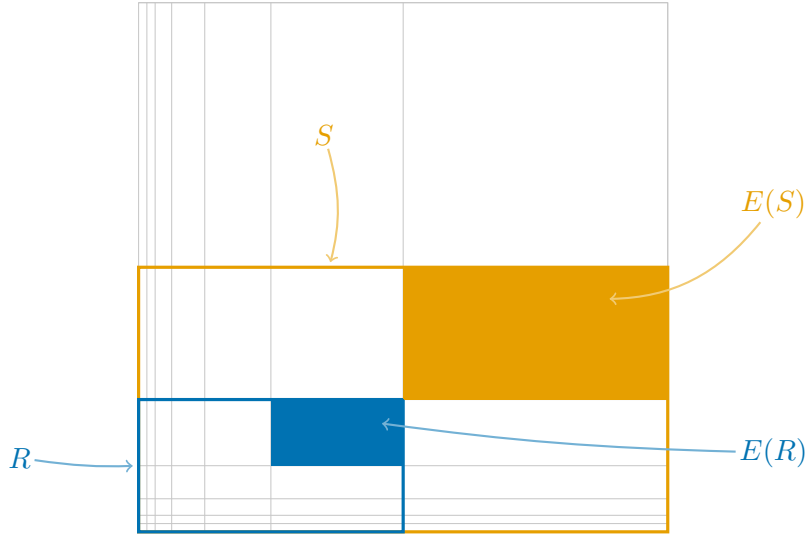
\begin{figure}[htpb!]
    \centering
    \begin{tikzpicture}[scale=7,
      lbl/.style={font=\large, inner sep=1.25pt},
      ptr/.style={->, thick, shorten >=2pt}]
      \foreach \t in {1, 0.5, 0.25, 0.125, 0.0625, 0.03125, 0.015625}{
        \draw[gray!45, very thin] (\t, 0) -- (\t, 1);
        \draw[gray!45, very thin] (0, \t) -- (1, \t);
      }
      \draw[gray!45, very thin] (0,0) rectangle (1,1);

      \draw[cborange, very thick] (0,0) rectangle (1, 0.5);
      \draw[cbblue, very thick] (0,0) rectangle (0.5, 0.25);

      \fill[cborange] (0.5, 0.25) rectangle (1, 0.5);
      \fill[cbblue] (0.25, 0.125) rectangle (0.5, 0.25);

      \node[cborange, lbl] (Slab) at (0.35, 0.75) {$S$};
      \draw[ptr, cborange!55] (Slab) to[bend left=15] (0.36, 0.5);

      \node[cborange, lbl] (ESlab) at (1.2, 0.62) {$E(S)$};
      \draw[ptr, cborange!55] (ESlab) to[bend left=25] (0.88, 0.44);

      \node[cbblue, lbl] (Rlab) at (-0.225, 0.14) {$R$};
      \draw[ptr, cbblue!55] (Rlab) to[bend right=5] (-0.001, 0.125);

      \node[cbblue, lbl] (ERlab) at (1.2, 0.15) {$E(R)$};
      \draw[ptr, cbblue!55] (ERlab) to[bend left=3] (0.45, 0.206);
    \end{tikzpicture}
    \caption{
      The family of dyadic rectangles anchored at the origin and contained in
      $[0,1)^2$.
      The gray grid is formed from the borders of the rectangles.
      Outlined: the rectangles
      $R = [0,2^{-1}) \times [0,2^{-2})$
      and $S = [0,1) \times [0,2^{-1})$,
      together with their associated sets $E(R)$ and $E(S)$.
    }
    \label{figure:chains_example}
  \end{figure}

  For every $x \in [0,2^{-N})^d$ we have the estimate
  \begin{align*}
    h_{\mathcal{F}}(x) &\geq \#\big(\big\{m \in \mathbb{Z}_{\geq 0}^d:\,
    m_i \leq N \text{ for all $i$} \big\}\big) \\
    &= (N+1)^d.
  \end{align*}
  Therefore, for all $N \geq 1$
  \begin{align}
    \label{sharp1:level_set_eta_sparse}
    |h_{\mathcal{F}} > N^d| \geq 2^{-Nd}.
  \end{align}

  This example is not completely satisfactory yet because it is not
  $\frac{1}{2}$-sparse.
  In order to pass to a $\frac{1}{2}$-sparse version we can use Theorem 4.1 of
  \cite{Rey2024}, which implies that there exist $L = \mathcal{O}(2^d)$
  subcollections $\mathcal{F}_i \subseteq \mathcal{F}$, each of which is
  $\frac{1}{2}$-sparse, and such that
  \begin{align*}
    \mathcal{F} = \bigcup_{i=1}^L \mathcal{F}_i.
  \end{align*}

  Since each $\mathcal{F}_i$ is $\frac{1}{2}$-sparse, by hypothesis we have, for
  all $1 \leq i \leq L$
  \begin{align*}
    \int_{\sh(\mathcal{F}_i)} \exp\big( c h_{\mathcal{F}_i}^\alpha \big)
    \leq C |\sh(\mathcal{F}_i)|.
  \end{align*}

  Let $K = L^{\max(0,\alpha-1)}$ then
  \begin{align} \label{sharp1:quasi}
    \bigg( \sum_{i=1}^L x_i \bigg)^\alpha \leq K \sum_{i=1}^L x_i^\alpha
  \end{align}
  is valid for all non-negative sequences $\{x_i\}$.

  Now, set $c_1 = \frac{c}{LK}$.
  Using \eqref{sharp1:quasi} and Hölder's inequality
  \begin{align*}
    \int_{\sh(\mathcal{F})} \exp\big( c_1 h_{\mathcal{F}}^\alpha \big)
    &\leq \int_{\sh(\mathcal{F})} \exp\Bigg(
      c_1 K \sum_{i=1}^L h_{\mathcal{F}_i}^\alpha
    \Bigg) \\
    &\leq \prod_{i=1}^L \Bigg(
      \int_{\sh(\mathcal{F})} \exp\big(c_1 LK h_{\mathcal{F}_i}^\alpha \big)
    \Bigg)^{\frac{1}{L}} \\
    &\leq \prod_{i=1}^L \Bigg( |\sh(\mathcal{F})| +
      \int_{\sh(\mathcal{F}_i)} \exp\big(c_1 LK h_{\mathcal{F}_i}^\alpha \big)
    \Bigg)^{\frac{1}{L}} \\
    &\leq \prod_{i=1}^L \big( 1 + C |\sh(\mathcal{F}_i)| \big)^{\frac{1}{L}}
    \\
    &\leq 1+C.
  \end{align*}
  Above, we have used that $\sh(\mathcal{F}) = [0,1)^d$, which has measure $1$.

  On the other hand, for all $N \geq 1$ we have, by
  \eqref{sharp1:level_set_eta_sparse}
  \begin{align*}
    \int_{\sh(\mathcal{F})} \exp(c_1 h_{\mathcal{F}}^{\alpha})
    &\geq 2^{-Nd} \exp(c_1 N^{d\alpha})
    = \exp\Big( c_1 N^{\alpha d} -\widetilde{c} N \Big),
  \end{align*}
  where the $\widetilde{c}$ is a positive dimensional constant.
  If $\alpha > \frac{1}{d}$ this expression tends to $\infty$ and we reach a
  contradiction.
\end{proof}

\subsubsection{The simplex construction}

The example contained in the next theorem shows that, under the antichain
condition, the exponent cannot be raised above $\frac{1}{d-1}$.

\begin{theorem} \label{sharpness:simplex}
  Fix $\alpha > 0$ and a dimension $d \geq 2$.
  Suppose there exist finite positive
  constants $c$ and $C$ such that
  \begin{align*}
    \int_{\sh(\mathcal{E})} \exp\big( c h_{\mathcal{E}}^\alpha \big)
    \leq C |\sh(\mathcal{E})|
  \end{align*}
  holds for all $\frac{1}{2}$-sparse antichains of dyadic $d$-dimensional
  rectangles. Then $\alpha \leq \frac{1}{d-1}$.
\end{theorem}
\begin{proof}
  Fix a large integer $N$ and consider the measure-one dyadic rectangles,
  containing the origin, and with no side shorter than $2^{-N}$:
  \begin{align*}
    \mathcal{E}_N := \Big\{ R_a:\,
      a \in \mathbb{Z}^d, \ \textstyle\sum_i a_i = 0, \ a_i \geq -N \Big\},
  \end{align*}
  where $R_a$ is the dyadic rectangle $R_a = \prod_{i=1}^d [0,2^{a_i})$.

  The exponent vectors $a$ range over the lattice points of a simplex:
  substituting $b_i := a_i + N$, they correspond to the nonnegative integer
  solutions of $b_1 + \dots + b_d = dN$, of which there are exactly
  $\binom{dN+d-1}{d-1}$ by the standard stars and bars count, so
  $\#(\mathcal{E}_N) = \binom{dN+d-1}{d-1} \asymp N^{d-1}$.

  The family is an antichain because all its members have the same volume,
  also it is $2^{-d}$-sparse since we can set, like before,
  \begin{align*}
    E(R_a) := \prod_{i=1}^d [2^{a_i-1}, 2^{a_i}),
  \end{align*}
  which are pairwise disjoint.
  In particular,
  \begin{align} \label{example:shadow}
    |\sh(\mathcal{E}_N)| \asymp N^{d-1}.
  \end{align}

  Like in the previous example, we now show how to obtain a
  $\frac{1}{2}$-sparse subcollection with similar properties.
  In this case we can just use the \eqref{p2} selection theorem.

  By Theorem \ref{zygmund:Selection}, we can extract a \eqref{p2} subcollection
  $\mathcal{F}_N \subseteq \mathcal{E}_N$ with
  $|\sh(\mathcal{F}_N)| \asymp |\sh(\mathcal{E}_N)|$.

  We can directly test with $\mathcal{F}_N$.
  Every rectangle $R_a \in
  \mathcal{E}_N$ contains the set $[0,2^{-N})^d$, thus
  \begin{align*}
    \int_{\sh(\mathcal{F}_N)} \exp\big(c h_{\mathcal{F}_N}^\alpha \big)
    &\geq \int_{[0,2^{-N})^d} \exp\big(c h_{\mathcal{F}_N}^\alpha \big) \\
    &= 2^{-Nd} \exp\big(c \#(\mathcal{F}_N)^\alpha \big).
  \end{align*}

  Since every element of $\mathcal{E}_N$ has unit measure, and it is sparse, we
  have
  \begin{align*}
    \#(\mathcal{F}_N) \asymp |\sh(\mathcal{F}_N)| \asymp N^{d-1}.
  \end{align*}
  Therefore, we have shown, for some finite positive constant $\widetilde{c}$,
  that for all $N \geq 1$
  \begin{align*}
    \int_{\sh(\mathcal{F}_N)} \exp\big(c h_{\mathcal{F}_N}^\alpha \big)
    &\geq \exp(\widetilde{c}(N^{\alpha(d-1)}-N)).
  \end{align*}
  Like before, we reach a contradiction if $\alpha > \frac{1}{d-1}$.
\end{proof}

\section{From height functions to strong-type bounds} \label{sec:strong_type}

In the previous section we have shown that, if $\mathcal{E}$ is a dyadic sparse
antichain of rectangles in $\mathbb{R}^2$, then the height function satisfies
an exponential-integrability estimate like
\begin{align*}
  \int_{\sh(\mathcal{E})} \exp(ch_{\mathcal{E}}) \lesssim |\sh(\mathcal{E})|.
\end{align*}

Here we show how this bound can be used to provide strong-type bounds for
the associated maximal functions.
We will prove the result in dimension two, but this part of the article
easily generalizes to higher dimensions assuming that the corresponding
dimensional tail bound also holds.
Additionally, we show that, for the maximal function estimate, the sparse
condition can be removed.

In particular, we will prove Theorem \ref{MT:SharpLpForAntichainM}. That is,
\begin{align} \label{strong_type:ineq}
  \|\mathcal{M}_{\mathcal{E}}f\|_{L^p(\mathbb{R}^2)} \lesssim p'
  \|f\|_{L^p(\mathbb{R}^2)}
\end{align}
for all $f \in L^p(\mathbb{R}^2)$ and all $1 < p < \infty$.
This inequality is sharp, as shown by testing it on simplices as in Theorem
\ref{sharpness:simplex}.
In particular, consider the dyadic antichain $\mathcal{E}$ given by
\begin{align*}
  \mathcal{E} = \{R_m :\, m \in \mathbb{Z} \},
\end{align*}
where $R_m = [0,2^m) \times [0,2^{-m})$. For every integer $N \geq 0$ define
\begin{align*}
  f_N = 2^{2N} \ind_{[0,2^{-N})^2}.
\end{align*}
Hence $\langle f_N \rangle_{R_m} = 1$ for all integers $m$ with $|m| \leq N$.
Setting $\mathcal{E}_N = \{R_m:\, |m| \leq N\}$ we have
\begin{align*}
  \int_{\mathbb{R}^2} \big(\mathcal{M}_{\mathcal{E}} f_N\big)^p
  &\geq \int_{\sh(\mathcal{E}_N)} \big( \mathcal{M}_{\mathcal{E}} f_N \big)^p \\
  &\geq |\sh(\mathcal{E}_N)| \asymp N,
\end{align*}
for all $N \geq 1$.
Above, we have used that $\mathcal{E}$ is $\frac{1}{2}$-sparse in the last step.
One easily computes $\|f_N\|_{L^p(\mathbb{R}^2)} = 2^{\frac{2N}{p'}}$, thus
\begin{align*}
  \frac{\|\mathcal{M}_{\mathcal{E}}f_N\|_{L^p(\mathbb{R}^2)}}
  {\|f_N\|_{L^p(\mathbb{R}^2)}} \gtrsim N^{\frac{1}{p}}2^{-\frac{2N}{p'}}
\end{align*}
for all $N \geq 1$.

Fix any $1 < p \leq 2$ and set $N = \lceil p' \rceil$.
Then, using the estimate above we obtain
\begin{align*}
  \frac{\|\mathcal{M}_{\mathcal{E}}f_N\|_{L^p(\mathbb{R}^2)}}
  {\|f_N\|_{L^p(\mathbb{R}^2)}} \gtrsim (p')^{\frac{1}{p}} = p'
  (p')^{-\frac{1}{p'}}.
\end{align*}
The factor $(p')^{\frac{1}{p'}}$ is bounded above and below by universal
constants when $1 < p \leq 2$, so the sharpness of \eqref{strong_type:ineq}
follows.

We now continue to the proof of Theorem \ref{MT:SharpLpForAntichainM}.
First, we remove the sparse condition using a result from \cite{CLR2026}, which
we restate adapted to our setting:
\begin{theorem}[Theorem 1.5 of \cite{CLR2026}]
  \label{CLR:SparseDomination}
  Given a countable collection $\mathcal{E}$ of sets of finite measure
  and $f \in L^\infty(\mathbb{R}^2)$ there exists a $\frac{1}{2}$-sparse
  subfamily $\mathcal{F} \subseteq \mathcal{E}$, such that
  \begin{align*}
    \mathcal{M}_{\mathcal{E}} f \leq
    2 \mathcal{P}_{\mathcal{E}}( \mathcal{M}_{\mathcal{F}} f).
  \end{align*}
\end{theorem}
The proof of this theorem is based on the selection algorithm from
\cite{CordobaFefferman1975}.
Before showing how it is used, let us explain the notation.
Above, $\mathcal{P}_{\mathcal{E}}$ denotes the ``maximal median''
operator associated to $\mathcal{E}$, which we define next.

For a given set $\Omega$ with finite measure, we define, for $r \in (0,1]$ and
measurable $f \geq 0$, the $r$-median of $f$ over $\Omega$ as
\begin{align*}
  \mathcal{P}^r(f;\, \Omega) = \inf\Big\{\lambda \geq  0:\,
    |\{x \in \Omega:\, f(x) > \lambda\}| \leq r|\Omega|
  \Big\}.
\end{align*}
Then, the maximal median operator associated to a family $\mathcal{E}$ is
defined as
\begin{align*}
  \mathcal{P}^r_{\mathcal{E}} f =
    \sup_{R \in \mathcal{E}} \ind_R \mathcal{P}^r(f;\, R).
\end{align*}
We omit the $r$ parameter when it is $\frac{1}{2}$:
\begin{align*}
  \mathcal{P}_{\mathcal{E}}f := \mathcal{P}^{\frac{1}{2}}_{\mathcal{E}}f.
\end{align*}

The operator $\mathcal{P}_{\mathcal{E}}$ is bounded on every $L^p$.
In fact we have the following stronger estimate.
Recall that $\mathcal{D}_d$ denotes the collection of all dyadic rectangles in
$\mathbb{R}^d$.
\begin{proposition} \label{smm}
  Let $0 < r \leq 1$, then for every measurable $f \geq 0$ and every real number
  $\lambda > 0$
  \begin{align*}
    |\mathcal{P}^r_{\mathcal{D}_d} f > \lambda| \leq \frac{4^d}{r^2}
    |f > \lambda|.
  \end{align*}
  In particular,
  \begin{align*}
    \|\mathcal{P}_{\mathcal{D}_d} f \|_{L^p(\mathbb{R}^d)} \leq
    4^{\frac{d+1}{p}}\|f\|_{L^p(\mathbb{R}^d)}.
  \end{align*}
\end{proposition}
\begin{proof}
  This follows from the identity
  \begin{align*}
    \{\mathcal{P}^r_{\mathcal{D}_d}f > \lambda\} =
    \{\mathcal{M}_{\mathcal{D}_d}(\ind_{f> \lambda}) > r \}.
  \end{align*}
  We can now use Chebyshev's inequality, and the $L^2$ bound of the
  dyadic strong maximal function, that is \eqref{doob} with $p=2$, to
  conclude
  \begin{align*}
    |\mathcal{P}^r_{\mathcal{D}_d}f > \lambda|
    &\leq \frac{1}{r^2} \big\| \mathcal{M}_{\mathcal{D}_d}(\ind_{f >
      \lambda}) \big\|_{L^2(\mathbb{R}^d)}^2
    \leq \frac{4^d}{r^2} |f > \lambda|.
  \end{align*}
\end{proof}

\begin{remark*}
  We have made no attempt to obtain the sharp constants in the previous result.
  In fact, the dependence on $r$ can be improved significantly to obtain
  \begin{align*}
    |\mathcal{P}_{\mathcal{D}_d}^r f > \lambda| \lesssim_d
     \frac{\log(e + r^{-1})^d}{r} |f > \lambda|
  \end{align*}
  using the weak-type of the strong maximal function.
  We shall not need such precise behavior on $r$, so we omit the details.
\end{remark*}

Since subcollections of antichains retain the antichain property,
for any dyadic antichain $\mathcal{E}$ and any bounded measurable
$f \geq 0$, there exists a $\frac{1}{2}$-\emph{sparse} dyadic antichain
$\mathcal{F}$ such that
\begin{align*}
  \|\mathcal{M}_{\mathcal{E}}f\|_{L^p(\mathbb{R}^2)} \leq
  128\|\mathcal{M}_{\mathcal{F}} f\|_{L^p(\mathbb{R}^2)}.
\end{align*}
Indeed, by Proposition \ref{smm}
\begin{align*}
  \|\mathcal{M}_{\mathcal{E}}f\|_{L^p(\mathbb{R}^2)} &\leq
  2\|\mathcal{P}_{\mathcal{E}}(
    \mathcal{M}_{\mathcal{F}} f)\|_{L^p(\mathbb{R}^2)} \\
  &\leq 2\|\mathcal{P}_{\mathcal{D}_2}(
    \mathcal{M}_{\mathcal{F}} f)\|_{L^p(\mathbb{R}^2)} \\
  &\leq 2 \cdot 4^{\frac{3}{p}}
    \|\mathcal{M}_{\mathcal{F}}f\|_{L^p(\mathbb{R}^2)}
  \leq 128 \|\mathcal{M}_{\mathcal{F}}f\|_{L^p(\mathbb{R}^2)}.
\end{align*}

For any collection $\mathcal{F}$ define the operator
\begin{align*}
  \mathcal{A}_{\mathcal{F}} f = \sum_{R \in \mathcal{F}} \langle f \rangle_R
  \ind_R.
\end{align*}
Summarizing: for every bounded measurable positive function $f$ and every dyadic
antichain $\mathcal{E}$ there exists a dyadic $\frac{1}{2}$-sparse antichain
$\mathcal{F}$ such that
\begin{align*}
  \|\mathcal{M}_{\mathcal{E}}f\|_{L^p(\mathbb{R}^2)}
  \leq 128 \|\mathcal{M}_{\mathcal{F}} f\|_{L^p(\mathbb{R}^2)}
  \leq 128 \|\mathcal{A}_{\mathcal{F}} f\|_{L^p(\mathbb{R}^2)},
\end{align*}
where we have used that
$\mathcal{M}_{\mathcal{F}}f \leq \mathcal{A}_{\mathcal{F}}|f|$ pointwise.
After a limiting argument, we conclude: for every two-dimensional dyadic
antichain $\mathcal{E}$
\begin{align*}
  \|\mathcal{M}_{\mathcal{E}}\|_{L^p(\mathbb{R}^2) \to L^p(\mathbb{R}^2)}
  &\leq 128\sup_{\mathcal{F}}
    \|\mathcal{A}_{\mathcal{F}}\|_{L^p(\mathbb{R}^2) \to L^p(\mathbb{R}^2)},
\end{align*}
where the supremum is taken over all finite two-dimensional dyadic
$\frac{1}{2}$-sparse antichains.

The strong maximal function $\mathcal{M}_{\mathcal{D}_2}$ dominates
$\mathcal{M}_{\mathcal{E}}$. Moreover, by \eqref{doob}, we already have
\begin{align*}
  \|\mathcal{M}_{\mathcal{D}_2}\|_{L^p(\mathbb{R}^2)} \leq 4
\end{align*}
for $p \geq 2$.

Therefore, to prove Theorem \ref{MT:SharpLpForAntichainM}  we only need to work
on the $1 < p \leq 2$ case.
Theorem \ref{MT:SharpLpForAntichainM} is therefore a consequence of the
following theorem.
\begin{theorem} \label{sharp_Lp_for_A}
  Let $\mathcal{E}$ be a finite $\frac{1}{2}$-sparse dyadic antichain of
  two-dimensional rectangles, then
  \begin{align} \label{sharp_Lp_for_A:ineq}
    \|\mathcal{A}_{\mathcal{E}}f\|_{L^p(\mathbb{R}^2)} \lesssim pp'
    \|f\|_{L^p(\mathbb{R}^2)} \qquad \text{for all }1 < p < \infty.
  \end{align}
\end{theorem}

Having reduced to the sparse setting, we can now apply the results of the
previous section to prove Theorem \ref{sharp_Lp_for_A}.
We begin with two simple applications of Theorem \ref{MT:ExpInt}.
We keep track of the constants in the proposition to make the following
arguments cleaner.
\begin{proposition} \label{height_function_bounds}
  There exists a finite positive constant $K$ such that for every
  two-dimensional $\frac{1}{2}$-sparse dyadic antichain $\mathcal{E}$
  \begin{enumerate}
    \item[(a)] For every $p \geq 1$
      \begin{align*}
        \|h_{\mathcal{E}}\|_{L^p} \leq K p |\sh(\mathcal{E})|^{\frac{1}{p}}.
      \end{align*}
    \item[(b)] For every measurable set $U$ with $0 < |U| < \infty$
      \begin{align*}
        \int_{U} h_{\mathcal{E}} \leq K |U|\log\bigg(e +
        \frac{|\sh(\mathcal{E})|}{|U|}\bigg).
      \end{align*}
  \end{enumerate}
\end{proposition}
\begin{proof}
  In this proof $c$ and $C$ are the constants in Theorem \ref{MT:ExpInt}.
  We may assume without loss of generality that $c \leq 1 \leq C$.

  We begin proving $(a)$. By elementary Calculus, $x^p \leq p^p e^x$, thus
  by Theorem \ref{MT:ExpInt}
  \begin{align*}
    \int h_{\mathcal{E}}^p \leq c^{-p}p^p \int_{\sh(\mathcal{E})}
    \exp(c h_{\mathcal{E}}) \leq C c^{-p} p^p |\sh(\mathcal{E})|,
  \end{align*}
  and hence
  \begin{align*}
    \|h_{\mathcal{E}}\|_{L^p} \leq C c^{-1} p
    |\sh(\mathcal{E})|^{\frac{1}{p}}.
  \end{align*}

  For $(b)$, fix a parameter $0 < t < \infty$, which we will optimize later.
  Using the layer-cake formula:
  \begin{align} \label{rtexp}
    \int_U h_{\mathcal{E}} &= \int_0^\infty |\{x \in U:\, h_{\mathcal{E}}(x) >
    \lambda \}| \, d\lambda \leq t|U| +
    \int_t^\infty |h_{\mathcal{E}} > \lambda| \, d\lambda.
  \end{align}

  Applying Chebyshev's inequality in Theorem \ref{MT:ExpInt}, we obtain
  \begin{align*}
    |h_{\mathcal{E}} > \lambda|
    \leq e^{-c \lambda} \int_{\sh(\mathcal{E})} \exp(c h_{\mathcal{E}}) \\
    \leq C e^{-c \lambda} |\sh(\mathcal{E})|.
  \end{align*}
  Plugging this estimate in \eqref{rtexp} yields
  \begin{align*}
    \int_{U} h_{\mathcal{E}} \leq t|U| +
    C c^{-1} |\sh(\mathcal{E})| e^{-c t}
    \leq C c^{-2} \Big( c t |U| + e^{-c t}|\sh(\mathcal{E})| \Big).
  \end{align*}
  Call $\Psi(s) := s|U| + e^{-s}|\sh(\mathcal{E})|$, then we have shown
  \begin{align*}
    \int_{U} h_{\mathcal{E}} \leq C c^{-2} \Psi(c t)
  \end{align*}
  for every $0 < t < \infty$. Choosing
  \begin{align*}
    t = c^{-1}\log\bigg( 1 + \frac{|\sh(\mathcal{E})|}{|U|} \bigg)
  \end{align*}
  gives
  \begin{align*}
    \int_{U} h_{\mathcal{E}} &\leq C c^{-2}
    |U|\bigg( 1 + \log\bigg(1+\frac{|\sh(\mathcal{E})|}{|U|} \bigg)\bigg) \\
    &\leq 2 C c^{-2}|U|\log\bigg( e + \frac{|\sh(\mathcal{E})|}{|U|} \bigg).
  \end{align*}

  It is clear that we can choose $K = 2C c^{-2}$ to arrive at the
  statement of the proposition.
\end{proof}

Before we continue, we first need the following elementary observation.
\begin{lemma} \label{inverse}
  Suppose $A \geq 1$ and $x > 0$, then
  \begin{align*}
    x \leq A \log(e+x) \implies x \leq 4A \log(e+A).
  \end{align*}
\end{lemma}
\begin{proof}
  Define the function $Q(x) = \frac{x}{\log(e+x)}$.
  This function is strictly increasing for $x \geq 0$, tends to $\infty$ as
  $x \to \infty$, and satisfies $Q(0) = 0$.
  Thus, it has an increasing inverse $Q^{-1}$.
  We wish to estimate $Q^{-1}(A)$.

  Fix a large constant $C$, to be determined later.
  Let us estimate the value of $Q$ at $C A\log(e+A)$:
  \begin{align*}
    Q(CA\log(e+A)) = A \frac{C\log(e+A)}{\log(e+CA\log(e+A))}.
  \end{align*}
  Set $t = \log(e+A)$ and observe that $t > 1$. Changing variables
  \begin{align*}
    \frac{C\log(e+A)}{\log(e+CA\log(e+A))} &= \frac{Ct}{\log(e+CAt)} \\
    &= \frac{Ct}{\log(e+Ct(e^t-e))} \\
    &\geq \frac{Ct}{\log(e+Ct e^t)} \\
    &\geq \frac{Ct}{\log(e+Ce^{2t})} \\
    &\geq \frac{Ct}{\log(e^{2t}(C+1))} \\
    &= \frac{Ct}{2t + \log(C+1)}.
  \end{align*}
  This last expression is increasing in $t$, and thus we have found
  \begin{align*}
    Q(CA\log(e+A)) \geq A \bigg( \frac{C}{2+\log(C+1)}\bigg).
  \end{align*}
  If we find $C$ large enough that
  \begin{align*}
    \frac{C}{2+\log(C+1)} \geq 1,
  \end{align*}
  we will have shown that $x \leq CA\log(e+A)$. The constant $C=4$ suffices.
\end{proof}

We now arrive at the main theorem in this section, which is a
restricted strong-type result.
\begin{theorem} \label{restricted_strong}
  There exists a finite positive constant $C$ such that
  for every two-dimensional $\frac{1}{2}$-sparse dyadic antichain $\mathcal{E}$,
  every $p \geq2$,
  and every measurable set $U$,
  we have
  \begin{align*}
    \|\mathcal{A}_{\mathcal{E}} \ind_U\|_{L^p(\mathbb{R}^2)}
    \leq C p|U|^{\frac{1}{p}}.
  \end{align*}
\end{theorem}
\begin{proof}
  By duality, it suffices to show, for every measurable positive $f$ with
  $\|f\|_{L^{p'}} = 1$, that
  \begin{align*}
    \int \mathcal{A}_{\mathcal{E}}(\ind_U) f
    = \sum_{R \in \mathcal{E}} \langle \ind_U \rangle_R \int_R f
    \leq C p|U|^{\frac{1}{p}}.
  \end{align*}
  For every integer $m \geq 1$ define
  \begin{align*}
    \mathcal{E}_m = \bigg\{R \in \mathcal{E}:\,
    2^{-m} < \frac{|U \cap R|}{|R|} \leq 2^{-m+1} \bigg\}.
  \end{align*}
  After removing the rectangles which do not intersect $U$, every
  $R \in \mathcal{E}$ is in exactly one $\mathcal{E}_m$.
  Also, every $\mathcal{E}_m$ is a $\frac{1}{2}$-sparse dyadic antichain, so we
  can estimate:
  \begin{align*}
    |\sh(\mathcal{E}_m)| &\leq \sum_{R \in \mathcal{E}_m} |R| \\
    &\leq 2^m \sum_{R \in \mathcal{E}_m} |U \cap R| \\
    &= 2^m \int_U h_{\mathcal{E}_m}.
  \end{align*}
  Let $K$ be the constant provided by Proposition \ref{height_function_bounds}.
  By part (b) applied to the family $\mathcal{E}_m$, we have
  \begin{align*}
    |\sh(\mathcal{E}_m)| \leq 2^m  K |U|\log\bigg( e +
    \frac{|\sh(\mathcal{E}_m)|}{|U|} \bigg).
  \end{align*}
  Set $A = 2^m K$ and $x = \frac{|\sh(\mathcal{E}_m)|}{|U|}$, then we have
  \begin{align*}
    x \leq A\log(e+x).
  \end{align*}
  By Lemma \ref{inverse}, we arrive at
  \begin{align*}
    x \leq 4A \log(e+A) \implies |\sh(\mathcal{E}_m)|
    &\leq 2^{m+2}K \log(e + 2^m K)|U| \\
    &\leq \widetilde{C} m 2^m |U|,
  \end{align*}
  for some constant $1 \leq \widetilde{C} < \infty$.

  Now, we can estimate
  \begin{align*}
    \sum_{R \in \mathcal{E}} \langle \ind_U \rangle_R \int_R f
    &= \sum_{m=1}^\infty \sum_{R \in \mathcal{E}_m}
      \langle \ind_U \rangle_R \int_R f \\
    &\leq 2\sum_{m=1}^\infty \sum_{R \in \mathcal{E}_m} 2^{-m} \int_R f \\
    &\leq 2 \sum_{m=1}^\infty 2^{-m} \int f h_{\mathcal{E}_m} \\
    &\leq 2 \sum_{m=1}^\infty 2^{-m} \|h_{\mathcal{E}_m}\|_{L^p}.
  \end{align*}
  By Proposition \ref{height_function_bounds} (a), we have
  \begin{align*}
    \|h_{\mathcal{E}_m}\|_{L^p} \leq K p|\sh(\mathcal{E}_m)|^{\frac{1}{p}},
  \end{align*}
  thus we conclude
  \begin{align*}
    \|\mathcal{A}_{\mathcal{E}} \ind_U\|_{L^p} &\leq 2 K p \sum_{m=1}^\infty
    2^{-m} \big( \widetilde{C} m2^m|U| \big)^{\frac{1}{p}} \\
    &\leq 2 K \widetilde{C} p |U|^{\frac{1}{p}}
      \sum_{m=1}^\infty 2^{-m} m^{\frac{1}{p}} 2^{\frac{m}{p}} \\
    &\leq 2 K \widetilde{C} p |U|^{\frac{1}{p}}
      \sum_{m=1}^\infty 2^{-m\big(1-\frac{1}{p}\big)} m.
  \end{align*}
  Recall that $p \geq 2$, so $1 - \frac{1}{p} \geq \frac{1}{2}$, hence
  \begin{align*}
    \|\mathcal{A}_{\mathcal{E}} \ind_U\|_{L^p}
    &\leq 2 K \widetilde{C} p |U|^{\frac{1}{p}}
    \sum_{m=1}^\infty 2^{-\frac{m}{2}} m
    \leq C p|U|^{\frac{1}{p}}
  \end{align*}
  for some constant $1 \leq C < \infty$.
\end{proof}

Finally, we are ready to prove Theorem \ref{sharp_Lp_for_A}.
Having established the restricted strong-type bound of Theorem
\ref{restricted_strong}, all that remains is an interpolation argument.
The precise statement we need is the following; in it, $(X, \mu)$ denotes
an arbitrary $\sigma$-finite measure space.
\begin{theorem} \label{interpolation}
  Let $T$ be a linear operator, acting on measurable functions on $X$,
  which is, a priori, bounded on $L^p(X, \mu)$ for all
  $2 \leq p < \infty$.
  Suppose that, for every measurable $U \subseteq X$ with $0 < \mu(U) < \infty$
  and every $q \geq 2$
  \begin{align} \label{interpolation:hypothesis}
    \|T(\ind_U)\|_{L^{q}} \leq Aq \mu(U)^{\frac{1}{q}}.
  \end{align}
  Suppose furthermore that we have the strong $L^2$ bound
  \begin{align} \label{interpolation:strongL2}
    \|Tf\|_{L^2} \leq B \|f\|_{L^2}.
  \end{align}
  Lastly, assume that $T$ satisfies the monotonicity property
  \begin{align*}
    0 \leq f \leq g \implies 0 \leq Tf \leq Tg.
  \end{align*}
  Then, for every $p \geq 3$
  \begin{align*}
    \|Tf\|_{L^p} \leq Cp\|f\|_{L^p},
  \end{align*}
  where the constant $C$ depends only on $A$ and $B$.
\end{theorem}
Qualitative versions of Theorem \ref{interpolation} are classical: it is
a variant of the interpolation theorem of Stein and Weiss for restricted
weak-type estimates \cite{SteinWeiss1959}; see
\cite[Corollary 1.4.22]{GrafakosClassical}, and the references therein,
for a statement closest to Theorem \ref{interpolation}.
The known statements, however, do not track how the constant depends on
$p$, and this dependence is critical for the proof of Theorem
\ref{sharp_Lp_for_A}.
Since the proof is short, we include the details in Appendix
\ref{app:interpolation}.

We now show that it already implies Theorem \ref{sharp_Lp_for_A}.
\begin{proof}[Proof of Theorem \ref{sharp_Lp_for_A}]

  Fix a $\frac{1}{2}$-sparse dyadic antichain of two-dimensional
  rectangles and set $T = \mathcal{A}_{\mathcal{E}}$.

  First, we note that $T$ is $L^p \to L^p$ bounded for all $1 < p < \infty$.
  By duality, take $g \geq 0$ with $\|g\|_{L^{p'}} = 1$.
  Then, by the $L^p$ norm bounds of the dyadic maximal function \eqref{doob}, we
  have
  \begin{align*}
    \int g\mathcal{A}_{\mathcal{E}}f &= \sum_{R \in \mathcal{E}} \langle f
    \rangle_R \langle g \rangle_R |R| \\
    &\leq 2 \int \mathcal{M}_{\mathcal{D}_2}f \, \mathcal{M}_{\mathcal{D}_2}g \\
    &\leq 2 \|f\|_{L^p} \|\mathcal{M}_{\mathcal{D}_2}\|_{L^p \to L^p}
    \|\mathcal{M}_{\mathcal{D}_2}\|_{L^{p'} \to L^{p'}} \\
    &\leq 2(pp')^2 \|f\|_{L^p}.
  \end{align*}

  That is, we have the \emph{suboptimal} a priori bound
  \begin{align} \label{suboptimal}
    \|\mathcal{A}_{\mathcal{E}}\|_{L^p \to L^p} \leq 2(pp')^2.
  \end{align}

  Next, notice that $\mathcal{A}_{\mathcal{E}}$ is self-adjoint, thus it
  suffices to prove \eqref{sharp_Lp_for_A:ineq} for $p \geq 2$.
  In fact, it suffices to treat $p \geq 3$ since, on the range
  $2 \leq p \leq 3$, the quantity $pp'$ is comparable to $1$, so the suboptimal
  bound \eqref{suboptimal} already gives \eqref{sharp_Lp_for_A:ineq} there.

  It remains to check that $T = \mathcal{A}_{\mathcal{E}}$ satisfies the
  hypotheses of Theorem \ref{interpolation}.
  The \emph{a priori} boundedness is \eqref{suboptimal}, and evaluating that
  bound at $p=2$ gives the strong $L^2$ estimate \eqref{interpolation:strongL2}.
  The restricted strong-type hypothesis \eqref{interpolation:hypothesis} is
  exactly Theorem \ref{restricted_strong}.
  Finally, $\mathcal{A}_{\mathcal{E}}$ is monotone since $0 \leq f \leq g$
  implies $0 \leq \langle f \rangle_R \leq \langle g \rangle_R$ for every $R$.
  Theorem \ref{interpolation} therefore applies, and yields
  \eqref{sharp_Lp_for_A:ineq}.
\end{proof}

\appendix

\section{Proof of the interpolation theorem} \label{app:interpolation}

In order to prove Theorem \ref{interpolation} we need the following Hardy
inequality. The reader can find this inequality, after changing variables
appropriately, in Theorem 330 of \cite{HLP1988}; we therefore omit its
proof.
\begin{proposition}[Hardy's inequality]
  Let $f \geq 0$ be a measurable function on $(0,\infty)$ and $p \geq 1$. Then
  \begin{align} \label{hardy:ineq}
    \int_0^\infty \bigg(
      \int_0^u f(s) \, ds
    \bigg)^{2p}u^{-p-1}\,du
    \leq4^p\int_0^\infty f(u)^{2p}\,u^{p-1}\,du.
  \end{align}
\end{proposition}

We are now ready to prove the interpolation theorem.
\begin{proof}[Proof of Theorem \ref{interpolation}]
  We can assume without loss of generality that $f \geq 0$ and that
  $A = B =  1$.
  By the layer-cake formula
  \begin{align*}
    \|Tf\|_{L^p}^p &= p \int_0^\infty |Tf > \lambda| \lambda^{p-1} \, d\lambda\\
    &= 2^p p \int_0^\infty |Tf > 2\lambda| \lambda^{p-1} \, d\lambda.
  \end{align*}
  For every $\lambda > 0$ split $f = a + b$, where
  $a = f \ind_{\{f > \gamma \lambda\}}$ and $\gamma = \frac{1}{16p}$.
  Then
  \begin{align*}
    (2^p p)^{-1}\|Tf\|_{L^p}^p &\leq
    \underbrace{\int_0^\infty |Ta > \lambda| \lambda^{p-1} \, d\lambda}_{
      \text{I}} +
    \underbrace{\int_0^\infty |Tb > \lambda| \lambda^{p-1} \, d\lambda}_{
      \text{II}}.
  \end{align*}

  We treat I and II separately. For I we use the strong $L^2$ bound
  \eqref{interpolation:strongL2}:
  \begin{align*}
    \text{I} &\leq \int_0^\infty
      \lambda^{-2} \bigg( \int a^2 \bigg) \lambda^{p-1} \, d\lambda \\
    &= \int_X f(x)^2 \int_0^{16pf(x)} \lambda^{p-3} \, d\lambda \, d\mu(x) \\
    &= \frac{(16p)^{p-2}}{p-2} \int_X f^{p} \, d\mu \\
    &= \frac{(16p)^{p-2}}{p-2} \|f\|_{L^p}^p \\
    &\leq (16p)^p \|f\|_{L^p}^p.
  \end{align*}

  To handle II, first notice that, for any measurable $g \geq 0$ we have
  \begin{align*}
    g = \sum_{m \in \mathbb{Z}} g \ind_{2^m < g \leq 2^{m+1}}
  \end{align*}
  thus, by the monotonicity property and the (qualitative) boundedness of $T$
  \begin{align*}
    Tg &\leq 2 \sum_{m \in \mathbb{Z}} 2^m T(\ind_{g > 2^m}).
  \end{align*}
  Now, by the restricted strong-type \eqref{interpolation:hypothesis}:
  \begin{align*}
    \|Tg\|_{L^p} &\leq 2p \sum_{m \in \mathbb{Z}} 2^m |g > 2^m|^{\frac{1}{p}} \\
    &\leq 4p \int_0^\infty |g > s|^{\frac{1}{p}} \, ds.
  \end{align*}
  We can use this inequality with exponent $2p$ and with $g = b$ to bound II:
  \begin{align*}
    \text{II} &= \int_0^\infty |Tb > \lambda| \lambda^{p-1} \, d\lambda \\
    &\leq \int_0^\infty \|Tb\|_{L^{2p}}^{2p} \lambda^{-p-1} \, d\lambda \\
    &\leq (8p)^{2p} \int_0^\infty \bigg(
      \int_0^\infty |b > s|^{\frac{1}{2p}} \, ds
    \bigg)^{2p}
    \lambda^{-p-1}d\lambda \\
    &\leq (8p)^{2p} \int_0^\infty \bigg(
      \int_0^{\gamma \lambda} |f > s|^{\frac{1}{2p}} \, ds
    \bigg)^{2p}
    \lambda^{-p-1}d\lambda \\
    &= (8p)^{2p} \gamma^p \int_0^\infty \bigg(
      \int_0^{\lambda} |f > s|^{\frac{1}{2p}} \, ds
    \bigg)^{2p}
    \lambda^{-p-1}d\lambda \\
    &= (4p)^p \int_0^\infty \bigg(
      \int_0^{\lambda} |f > s|^{\frac{1}{2p}} \, ds
    \bigg)^{2p}
    \lambda^{-p-1}d\lambda
    =: \circledast.
  \end{align*}
  By Hardy's inequality \eqref{hardy:ineq}
  \begin{align*}
    \circledast &\leq (16p)^p \int_0^\infty
    |f > \lambda| \lambda^{p-1} \, d\lambda \\
    &\leq (16p)^p \|f\|_{L^p}^p.
  \end{align*}

  Putting everything together:

  \begin{align*}
    \|Tf\|_{L^p} &\leq (2^p p)^{\frac{1}{p}} \big(2(16p)^p\big)^{\frac{1}{p}}
    \|f\|_{L^p} \\
    &\leq Cp \|f\|_{L^p}.
  \end{align*}
\end{proof}

\bibliography{zygmund}
\bibliographystyle{abbrv}

\end{document}